\documentclass[12pt]{article}
\usepackage{amssymb}
\usepackage{amsmath}
\usepackage{amsthm}

\numberwithin{equation}{section}
\usepackage{graphicx} 
\usepackage{subfigure}
\newtheorem{theorem}{Theorem}[section]
\newtheorem{lemma}{Lemma}[section]
\newtheorem{proposition}{Proposition}[section]

\usepackage{epstopdf}
\usepackage{float}
\usepackage{epsf,epsfig,subfigure}
\date{}

\newcommand{\di}{{\rm d}}

\begin{document}
	\title{Propagation Dynamics for a Spatially Periodic Integrodifference Competition Model\thanks{Research supported in part by the NSERC of
			Canada.}}
	
	\author{Ruiwen Wu and Xiao-Qiang Zhao\\
		Department of Mathematics and Statistics\\
		Memorial University of Newfoundland\\
		St. John's, NL A1C  5S7, Canada\\
		E-mail:\, rw2403@mun.ca \, \,  zhao@mun.ca}
	\maketitle
	\baselineskip 0.24in
	\noindent {\bf Abstract.}  
In this paper, we study the propagation dynamics for a class of  integrodifference competition models in a periodic habitat. 
An interesting feature of such a system is that multiple spreading speeds can be observed, which biologically means different species may have different spreading speeds. We  show that the model system admits a single spreading speed, and it coincides with the minimal wave speed of the spatially periodic traveling waves. A set of sufficient conditions for linear determinacy of the spreading speed is also given.
	
	\smallskip
	

	\smallskip
	
\section{Introduction}

Competition exists widely in the multispecies interaction. One of the crucial concepts on describing the competitive dynamics is called the competition exclusion principle, also referred to as Gause's Law \cite{kot2001}, which states that if two species attempting to occupy the limited resources cannot coexist, then one species will drive out the other. Competition exclusion provides useful insights on ecological balance, for instance, beneficial invasion can be introduced in pest control. Among those theoretical models, a spatially-independent difference system is  the following Leslie/Gower competition model:
\begin{eqnarray}\label{VL1}
	& &p_{n+1}=
	\frac{r_1p_n}{1+\frac{r_1-1}{C_1}(p_n+a_1q_n)}, \\ 
	& &q_{n+1}=
	\frac{r_2q_n}{1+\frac{r_2-1}{C_2}(q_n+a_2p_n)}, \nonumber
\end{eqnarray}
where $p_n$ and $q_n$ are the population densities of two competing species at time $n$. The competition between two species is governed by Beverton-Holt dynamics. $r_i \ (r_i>1)$, $C_i$ and $a_i$ are growth rates, carrying capacity of $i$-th species $(i=1,2)$, and interspecific competition coefficients, respectively. The global dynamics of system \eqref{VL1}
was discussed by Cushing et al. (see \cite[Lemma 2]{cushing}), and the competition exclusion occurs if interspecific competition is too large \cite{cushing}.

In nature, real species are usually spatially extended, and hence, the effects of dispersal processes are of high interest in spatial ecology. In well-known diffusion models, growth is usually assumed to occur at the same time with dispersal. However, in many situations such as annual and perennial species plants, migrating bird species, growth and dispersal are in distinct stages. Thus, integrodifference equations, which are continuous in space and discrete in time, become more realistic and popular. Kot and Schaffer \cite{Kot1986} first applied integrodifference euqations to population modeling. Since then, the study of integrodifference equations in ecology gained a lot of attention, see, e.g., \cite{Ding13,Ding15,hardin1990,latore1998,luscher2008,neubent2000,neubent1995,van1997}. Mathematical investigations includes the study of traveling waves \cite{Hsu2,Kot1992,Kot1996,Wein2008} and analytical approximation schemes \cite{Gilbert}. Recently, Zhou and Kot \cite{ZHOU2011} considered an integrodifference euqation with shifting species ranges subject to climate changes, and  Zhou and Fagan \cite{ZHOU2017} investigated a single-species integrodifference model with time-varying size. 

Apart from population dispersal, how species interact with space is another important topic in spatial ecology, since most lanscapes are heterogeneous. Travelling waves and spreading speeds are commonly used to explore the propagation dynamics. Shigesada et al. \cite{SKT} first studied the spreading speeds for single-species continuous-time model in a periodic patchy habitat. Later, Kawasaki and Shigesada \cite{Kawa2007} extended the work to discrete-time models. A general theory of travelling waves and spreading speeds in a periodic habitat was developed by Weinberger \cite{Wein02}, Liang and Zhao \cite{Liang2}, and Fang and Zhao \cite{FZ}. Recently, Yu and Zhao \cite{Yu} studied the propagation phenomena of a two species reaction-advection-diffusion competition model in a periodic habitat by appealing to the abstract results in \cite{FZ,Liang2}. 

Naturally, system \eqref{VL1} can be extended to the following spatial model:
\begin{eqnarray}\label{VL}
	& &p_{n+1}(x)=
	\int_{\mathbb{R}}\frac{r_1(y)p_n(y)}{1+b_1(y)(p_n(y)+a_1(y)q_n(y))}k_1(x,y)\di y, \\ 
	& &q_{n+1}(x)=
	\int_{\mathbb{R}}\frac{r_2(y)q_n(y)}{1+b_2(y)(q_n(y)+a_2(y)p_n(y))}k_2(x,y)\di y, \ x\in\mathbb{R}, \nonumber
\end{eqnarray}
where
\[b_i(x)=\frac{r_i(x)-1}{C_i(x)},\ (i=1,2),\]
$p_n(x)$ and $q_n(x)$ are the population densities of two competing species at time $n$ and location $x$. $k_i(x,y)$ is the probability density function for the destination $x$ of individuals from $y$ of $i$-th species $(i=1,2)$. 
As mentioned in \cite{ZHOU2011},  both population persistence and invasion dynamics are worthy to be considered. For  system \eqref{VL} with distance-dependent kernel, i.e., $r_i(x)$, $C_i(x)$, $a_i(x) \ (i=1,2)$ are constant and $k_i(x,y)=k_i(x-y)$, the propagation phenomena has been 
investigated by Lewis, Li and Weinberger \cite{Lewis2002} in the monostable case, and by Zhang and Zhao \cite{Zhang2012} in the  bistable case. Samia and Lutscher \cite{Lutscher2010} also studied the competitive coexistence for system \eqref{VL} in a patchy habitat in two specific cases: competitive-ability-varying one and carrying-capacity-varying one. 

Motivated by these works, we are interested in  the invasion dynamics of system \eqref{VL} in the case of competition exclusion.
In order to consider a periodic habitat, the coefficients $r(x)$, $C(x)$, $a(x)$ and $k(x,y)$ are assumed to be periodic functions of space. Therefore, we need the following assumptions for $r$, $C$, $a$ and $k(x,y)$:
\begin{enumerate}
	\item[(K1)] The habitat is $L$-periodic with some positive number $L$ such that $r(x)>1, \ C(x)>0,\ a(x)>0,$ and 
	\[r(x+L)=r(x), \ C(x+L)=C(x), \ a(x+L)=a(x), \  \forall x \in \mathbb{R}.\]
	\item[(K2)] The dispersal kenrnel $k(x,y)$ has the following properties:
	\begin{itemize}
		\item[(i)] $k(x+L,y+L)=k(x,y),  \ \forall x,y \in \mathbb{R}.$
		\item[(ii)] For each $x$, $k(x,y)$ satiesfies $k(x,y)\geqslant 0$ and 
		$\displaystyle\int_{-\infty}^{+\infty}k(x,y)\di y<\infty$, and for each $y$,
		\[\int_{-\infty}^{+\infty}k(x,y)\di x=1.\]
		\item[(iii)] $k(x,y)$ is lower semicontinuous in the sense that for each $(x_0,y_0)$ and each $\varepsilon>0$
		there is a positive number $\delta(x_0,y_0,\varepsilon)$ such that 
		$k(x,y)\geqslant k(x_0,y_0)-\varepsilon$ whenever $|x-x_0|+|y-y_0|\leqslant\delta(x_0,y_0,\varepsilon)$.
		\item[(iv)] There are an integer $\xi$ and a positive integer $\eta$ with the following properties:
		For every a with $|\alpha|\leqslant L/2$, and for every $\beta$ with $|\beta-\xi L| \leqslant L$, there
		is a $\eta+1$-tuple of numbers $x_0$, $x_1$, $\cdot\cdot\cdot$, $x_{\eta}$ such that $x_0=\alpha$, $x_{\eta}=\beta$,
		and $k(x_j,x_{j-1})>0$ for $j=1,2,\cdot\cdot\cdot, \eta$.	
		\item[(v)] $k(x, y)$ is uniformly $L_1$-continuous in $x$ in the sense that
		\[\lim\limits_{h\rightarrow0}\int_{-\infty}^{+\infty}|k(x+h,y)-k(x,y)|\di y=0, \ \text{uniformly in} \ x \in\mathbb{R}. \]
		\item[(vi)] There exists $\mu^*>0$ such that for fixed $\mu \in [0,\mu^*)$, $k_i(x,y)$ satisfies
		\[\int_{-\infty}^{+\infty}k(x,y)e^{-\mu(y-x)}\di y<\infty,\]
		where $\mu^*>0$ is the abscissa of convergence and it may be infinity.
		
	\end{itemize}
\end{enumerate}
 
 We remark that (K2(v)) is used to guarantee the equicontinuity of the integral operator $Q$, which is generated by system (\ref{VL}) \cite[[Hypotheses 2.1(iv)]{Wein2008}, then further to prove the compactness of $Q$ and its Fr\'{e}chet derivative $DQ(0)$ \cite[Lemma 2.1]{Ding13}. (K2(i)-(iv)) are needed in the proof of the strong positivity of $DQ(0)$. We write $r_i(x), \ C_i(x), \ a_i(x)$ and $k_i(x,y)$ to denote the relvant parametes of $i$-th species $(i=1,2)$.

The purpose of this paper is to study the spatially periodic travelling waves and spreading speeds for system \eqref{VL}. We first prove the existence of periodic steady states $(p^*(x),0)$ and $(0,q^*(x))$, and globally attractivity of $(p^*(x),0)$ for system (\ref{VL}) with periodic initial values under appropriate assumptions.
Note that the steady state $(0,0)$ is between $(p^*(x),0)$ and $(0,q^*(x))$ with respect to the competitive ordering, which implies the possibility of multiple spreading speeds. Such a situation was also pointed out in \cite{li2005}. By appealing to the theory developed in \cite{FZ} which allows the existence of  boundary fixed points between two ordered unstable and stable fixed points, we are able to prove 
the existence of the rightward spatially periodic travelling waves connecting $(p^*(x),0)$ to $(0,q^*(x))$, and show that the system has a single spreading speed under some appropriate conditions. We also obtain a set of sufficient conditions for the rightward spreading speed to be linearly determinate.

The rest of this paper is organized as follows. The existence of two semi-trivial periodic steady states and the global attractivity of one semi-trivial periodic steady state are investigated in Section 2. In Section 3, we present the results on spatially periodic travelling waves and the existence of single spreading speed. We otain the linear determancy for the spreading speed in Section 4.   
In Section 5, we apply the obtained results to a patchy senario in which the carry capacity is spatially varying, and we also provide a simple example to verify the linear determancy condition.
Some numerical simulations are presented to illustrate the analytic results.

\section{The periodic initial value problem}
In this section, we study the global dynamics of the spatially periodic integrodifference competition system with the periodic initial values.

Let $Y$ be the set of all continuous and $L$-period functions from $\mathbb{R}$ to $\mathbb{R}$, and $Y_+=\{\psi\in Y: \ \psi(x)\ge0,\forall x\in\mathbb{R}\}$. Equip $Y$ with the maximum norm $\|\phi\|_Y$, that is,
$\|\phi\|_Y=\max_{x\in \mathbb{R}}|\phi(x)|.$ Then $(Y,Y_+)$ is a strongly ordered Banach lattice.
Assume that $L$-periodic functions $r\in C(\mathbb{R})$ satisfying $r(x)>1, \ \forall x \in \mathbb{R}$. We can define
\[\check{L}\phi(x)= \int_{\mathbb{R}}r(y)\phi(y)k(x,y)\di y, \quad x\in\mathbb{R}.\]
By the arguments similar to those in \cite{Wein2008}, it is easy to verify $(\check{L})^\eta$ is strongly positive, where $\eta$ is the positive integer in (K2(iv)). By \cite[Lemma 3.1]{Liang}, we know that the spectral radius $\rho(\check{L})$ is a single eigenvalue of $\check{L}$, with an associated strongly positive $L$-periodic eigenfunction $\phi(x)$. 
It follows that the scalar periodic eigenvalue problem 
\begin{eqnarray}\label{VLpep1}
& & \lambda \phi(x)= \int_{\mathbb{R}}r(y)\phi(y)k(x,y)\di y, \quad x\in\mathbb{R},\nonumber\\ 
& &\phi(x+L)=\phi(x),\quad x\in\mathbb{R}
\end{eqnarray}
admits a principal eigenvalue $\lambda(k,r)=\rho(\check{L})$ associated with a strongly positive $L$-periodic eigenfunction $\phi(x)$.  
As a consequence of \cite[Theorem 2.3.4]{Zhaobook},  we have the following result.
\begin{proposition}\label{VLexistence}Assume that $L$-periodic functions $b(x), \ r(x),\ k(x,y)$ satisfied (K1) and (K2).
	Let $p_n(x,\phi)$ be the unique solution of the following equation:
	\begin{eqnarray}\label{VLseq}
	& & p_{n+1}(x)=
	\int_{\mathbb{R}}\frac{r(y)p_n(y)}{1+b(y)p_n(y)}k(x,y)\di y\  x\in\mathbb{R},\nonumber\\
	& & p_0(x)=\phi(x)\in Y_+, \quad x\in \mathbb{R}.
	\end{eqnarray} 
	Then the following statements are valid:
	\begin{enumerate}
		\item[(i)] If $\lambda(k,r)\leqslant 1$, then $p_n(x)=0$ is globally asymptotically stable with
		respect to initial values in $Y_+$;
		\item[(ii)] If $\lambda(k,r)>1$, then \eqref{VLseq} admits a unique positive $L$-periodic steady state
		$p^*(x)$, and it is globally asymptotically stable with respect to initial values
		in $Y_+\backslash\{0\}$.
	\end{enumerate}
\end{proposition}  
Let $\mathbb{P}=PC(\mathbb{R},\mathbb{R}^2)$ be the set of all continuous and $L$-periodic functions from $\mathbb{R}$ to $\mathbb{R}^2$, and $\mathbb{P}_+=\{\psi\in \mathbb{P}: \ \psi(x)\ge0,\forall x\in\mathbb{R}\}$. 
Then $\mathbb{P}_+$ is a closed cone of $\mathbb{P}$ and induces a partial ordering on $\mathbb{P}$. Moreover, we introduce a norm $\|\phi\|_\mathbb{P}$ by
\[\|\phi\|_\mathbb{P}=\max_{x\in \mathbb{R}}|\phi(x)|.\]
It then follows that $(\mathbb{P},\|\phi\|_\mathbb{P})$ is a Banach lattice. For any $\varphi\in\mathbb{P}_+$, system \eqref{VL} has a unique nonnegative solution $(p_n(\cdot,\varphi),q_n(\cdot,\varphi))\in\mathbb{P}_+$.

In view of Proposition \ref{VLexistence}, there exists two positive $L$-periodic functions $p^*(x)$ and $q^*(x)$ such that $E_1:=(p^*(x),0)$, $E_2:=(0,q^*(x))$ are semi-trivial steady states of system \eqref{VL} provided that $\lambda(k_i,r_i)>1\ (i=1,2).$
Since we mainly concern about the case of the competition exclusion, we impose the following conditions on system \eqref{VL}:
\begin{enumerate}
	\item[(H1)]$\lambda(k_i,r_i)>1\ (i=1,2).$
	\item[(H2)]$\lambda\Big(k_1,\displaystyle\frac{r_1}{1+b_1a_1q^*}\Big)>1.$
	\item[(H3)]System \eqref{VL} has no steady state in Int$(\mathbb{P}_+)$.
\end{enumerate}

Note that (H1) guarantees the existence of two semi-trivial steady states of system \eqref{VL}. (H2) implies that $(0,q^*(x))$ is unstable.
Under the assumption (H1)--(H3), there are three steady states in $\mathbb{P}_+$: $E_0=(0,0)$, $E_1:=(p^*(x),0)$, and $E_2:=(0,q^*(x))$. Next, we use the theory developed in \cite{Hsu} for abstract competitive systems (see also \cite{Hess2}) to prove the global attractivity of $E_1$.
\begin{theorem}\label{VLEQ}
	Assume that (K1)--(K2), and (H1)--(H3) hold. Then $E_1=(p^*(x), 0)$ is globally asymptotically stable for initial values $\phi=(\phi_1,\phi_2)$ in $\mathbb{P}_+$ with $\phi_1\not\equiv 0$.
\end{theorem}
\begin{proof}
	Let $P_n(x,\phi)=(p_n(x,\phi),q_n(x,\phi))$ be the solution of system \eqref{VL} with $p_0(x)=\phi(x)$. Since (H2) holds, we can fix $\varepsilon_0\in \Bigg(0,1-\displaystyle\frac{1}{\lambda(k_1,\frac{r_1}{1+b_1a_1q^*})}\Bigg)$. By the uniform continuity of 
	\[F(x,P):=\displaystyle\frac{r_1}{1+b_1(p+a_1q)}\] 
	on the set $\mathbb{R}\times[0,1]\times[0,m]$, where $M=\max\limits_{x\in \mathbb{R}}q^*(x)+1$, it follows that there exists $\delta_0\in(0,1)$ such that 
	\[|F(x,P^{(1)})-F(x,P^{(2)})|<\varepsilon_0\cdot A,\quad \forall P^{(1)}=(p^{(1)},q^{(1)}),P^{(2)}=(p^{(2)},q^{(2)})\in[0,1]\times[0,m],\] 
	provided that $|p^{(1)}-p^{(2)}|<\delta_0$ and $|q^{(1)}-q^{(2)}|<\delta_0$, where $A=\min\limits_{x\in \mathbb{R}}\displaystyle\frac{r_1(x)}{1+b_1(x)a_1(x)q^*(x)}$, $A>0$. Then we have the following claim.
	
	\noindent {\it Claim}. For all $\phi\in\mathbb{P}_+$ with $\phi_1\not\equiv0$, there holds \[\displaystyle\lim\sup_{n\rightarrow\infty}\|(p_n(x,\phi),q_n(x,\phi))-(0,q^*(x))\|_\mathbb{P}\ge\delta_0.\]
	
	Suppose, by way of contradiction, that $\displaystyle\lim\sup_{n\rightarrow\infty}\|(p_n(x,\phi),q_n(x,\phi))-(0,q^*(x))\|_\mathbb{P}<\delta_0$ for some $\hat\phi\in\mathbb{P}_+$ with $\hat\phi_1\not\equiv0.$ Then there exists $n_0>0$ such that $$\|p_n(\cdot,\hat\phi)\|_{Y}<\delta_0,\  \|q_n(\cdot,\hat\phi)-q^*(\cdot)\|_Y<\delta_0,\  \forall n\ge n_0.$$
	Consequently, we have
	\[F(x, P_n(x,\hat\phi))>F(x,(0,q^*(x)))-\varepsilon_0\cdot A=(1-\varepsilon_0)F(x,(0,q^*(x))),\quad  \forall n\ge n_0,\  x\in \mathbb{R}.\]
	Let $\psi_1(x)$ be a positive eigenfunction corresponding to the principal eigenvalue $\lambda\Big(k_1,\displaystyle\frac{r_1}{1+b_1a_1q^*}\Big)$. Then $\psi_1(x)$ satisfies
	\begin{align}\label{VLeq1}
	&  \lambda\Big(k_1,\displaystyle\frac{r_1}{1+b_1a_1q^*}\Big)\psi_1\!=\!\int_{\mathbb{R}}\frac{r_1(y)}{1+b_1(y)a_1(y)q^*(y)}\psi_1(y)k_1(x,y)\di y,\quad x\in\mathbb{R}, \nonumber\\ 
	& \psi_1(x+L)=\psi_1(x),\quad x\in\mathbb{R}.
	\end{align}
	Since $p_0(\cdot)=\hat{\phi_1}\geqslant\not\equiv0$, the comparison principle, as applied to the first equation in system \eqref{VL}, implies that $p_{n_0}(x,\hat\phi)>0,\ \forall x\in\mathbb{R}$. Then there exists small $\eta>0$ such that $p_{n_0}(\cdot)\ge\eta\psi_1\gg0$. Thus, $p_n(x,\hat\phi)$ satisfies
	\begin{eqnarray}\label{VLeq2}
	& &  p_{n+1}(x)\ge
	\int_{\mathbb{R}}\frac{r_1(y)(1-\varepsilon_0)}{1+b_1(y)a_1(y)q^*(y)}p_n(y)k_1(x,y)\di y, \quad \forall n>n_0,\ x\in\mathbb{R},\nonumber\\ 
	& &p_{n_0}(\cdot)\ge\eta\psi_1.
	\end{eqnarray}
	In view of \eqref{VLeq1}, it easily follows that 
	$\bar{p}_n(\cdot)=\eta [(1-\varepsilon_0)\lambda(k_1,\frac{r_1}{1+b_1a_1q^*})]^{(n-n_0)}\psi_1$ satisfies
	\begin{eqnarray}\label{VLeq3}
	& & \bar{p}_n(x)= \int_{\mathbb{R}}\frac{r_1(y)(1-\varepsilon_0)}{1+b_1(y)a_1(y)q^*(y)}\tilde{p}_n(y)k_1(x,y)\di y, \quad n>n_0, x\in\mathbb{R},\nonumber\\ 
	& &\bar{p}_{n_0}(\cdot)=\eta\psi_1.
	\end{eqnarray}
	By \eqref{VLeq2} and \eqref{VLeq3}, together with the standard comparison principle, it follows that  
	\[p_n(\cdot,\hat\phi)\ge\eta \Big[(1-\varepsilon_0)\lambda\Big(k_1,\displaystyle\frac{r_1}{1+b_1a_1q^*}\Big)\Big]^{(n-n_0)}\psi_1,\quad \forall  n\ge n_0.\] 
	Letting $n\rightarrow \infty$, we see that $p_n(\cdot,\hat{\phi})$ is unbounded, a contradiction.
	
	By the above claim and (H3), we exclude possibility (a) and (c) in \cite[Theorem A]{Hsu}. Since $E_2$ is repellent in some neighborhood of itself, it follows from \cite[Theorem A]{Hsu} that $E_1$ is globally asymptotically attractive.    
\end{proof}

\section{Spreading speeds and traveling waves}
In this section, we study the spreading speeds and spatially periodic traveling waves for system \eqref{VL}. 
By a change of variables $u_n=p_n, v_n=q^*(x)-q_n$, we transform system \eqref{VL} into the following cooperative system:
\begin{eqnarray}\label{NModel}
& &u_{n+1}(x)=
\int_{\mathbb{R}}\frac{r_1(y)u_n(y)}{1+b_1(y)(u_n(y)+a_1(y)(q^*(y)-v_n(y))}k_1(x,y)\di y, \\ 
& &v_{n+1}(x)=
\int_{\mathbb{R}}\frac{r_2(y)}{1+b_2(y)q^*(y)}\cdot \frac{b_2(y)a_2(y)q^*(y)u_n(y)+v_n(y)}{1+b_2(y)(q^*(y)+a_2(y)u_n(y)-v_n(y))}k_2(x,y)\di y. \nonumber
\end{eqnarray}
Then three steady states of \eqref{VL} become
\[\hat E_0=(0,q^*(x)),\ \hat E_1=(p^*(x),q^*(x)),\ \hat E_2=(0,0).\] 

Let $\mathcal{C}$ be the set of all bounded and continuous functions
from $\mathbb{R}$ to $\mathbb{R}^2$ and $\mathcal{C}_+=\{\phi\in\mathcal{C}:\phi(x)\geqslant 0,\ \forall x\in \mathbb{R}\}$.  Assume that $\beta$ is a  strongly positive $L$-periodic continuous function from $\mathbb{R}$ to $\mathbb{R}^2$. Set 
\begin{small}$$\mathcal{C}_{\beta}=\{u\in \mathcal{C}:\, 0\leqslant u(x)\leqslant \beta(x),\ \forall x\in \mathbb{R}\},\  \mathcal{C}^{per}_{\beta}=\{u\in \mathcal{C_\beta}:\,  u(x)=u(x+L),\ \forall x\in \mathbb{R}\}.$$\end{small} 
Let $X=C([0,L],\mathbb{R}^2)$ equipped with the maximum norm $|\cdot|_X$, $X_+=C([0,L],\mathbb{R}_+^2)$, $$X_{\beta}=\{u\in X:\ 0\leqslant u(x)\leqslant{\beta}(x),\  \forall x\in[0,L]\}\  \text{and}\  \overline{X}_{\beta}=\{u\in X_{\beta}: u(0)=u(L)\}.$$ Let $BC(\mathbb{R}, X)$ be the set of all continuous and bounded functions from $\mathbb{R}$ to $X$. Define \begin{small}$$\mathcal{X}=\{v\in BC(\mathbb{R},X):v(s)(L)=v(s+L)(0),\forall s\in \mathbb{R}\},  \mathcal{X}_+=\{v\in \mathcal{X}:v(s)\in X_+,\forall s\in \mathbb{R}\},$$\end{small} and  
$$\mathcal{X}_{\beta}=\{v\in BC(\mathbb{R},X_{\beta}):v(s)(L)=v(s+L)(0),\forall s\in \mathbb{R}\}.$$ 
We equip $\mathcal{C}$ with the compact open topology, that is, $u_m\to
u$ in $\mathcal{C}$ means that the sequence of $u_m(s)$ converges to $u(s)$ in $\mathbb{R}^m$ uniformly for $s$ in any compact set. We equip $\mathcal{C}$ with the norm $\|\cdot\|_\mathcal{C}$ given by \[\|u\|_{\mathcal{C}}=\sum\limits_{k=1}^{\infty}\frac{\max_{|x|\leqslant k}|u(x)|}{2^k},\ \forall u\in\mathcal{C},\] 
where $|\cdot|$ denotes the usual norm in $\mathbb{R}^m$, and
\[\|u\|_{\mathcal{X}}=\sum\limits_{k=1}^{\infty}\frac{\max_{|x|\leqslant k}|u(x)|_X}{2^k},\ \forall u\in\mathcal{X}.\] 

Let $\beta(\cdot)=(p^*(\cdot),q^*(\cdot))$, and
%
$Q$ be a map on $\mathcal{C}_{\beta}$ with $Q[0]=0$ and $Q[\beta]=\beta$. We say that $V(x-cn,x)$ is an $L$-periodic rightward traveling wave of $Q$ if $V(\cdot+a,\cdot)\in \mathcal{C}_\beta$,\ $\forall a\in \mathbb{R}$, $Q^n[V(\cdot,\cdot)](x)=V(x-cn,x)$, $\forall n\ge0$, and $V(\xi,x)$ is an $L$-periodic function in $x$ for any fixed $\xi\in\mathbb{R}$. Moreover, we say that $V(\xi,x)$ connects $\beta$ to $0$ if $\lim_{\xi\rightarrow -\infty}|V(\xi,x)-\beta(x)|=0$ and $\lim_{\xi\rightarrow +\infty}|V(\xi,x)|=0$ uniformly for $x\in\mathbb{R}$. According to \cite{Yu}, we need the following assumptions:
\begin{itemize}
	\item[(A1)] $Q$ is $L$-periodic, that is, $\mathcal{T}_a[Q[u]]
	=Q[\mathcal{T}_a[u]],\quad  \forall u\in \mathcal{C}_{\beta},\,
	a\in L\mathbb{Z}$.
	\item[(A2)] $Q:\, \mathcal{C}_{\beta} \to \mathcal{C}_{\beta}$ is continuous
	with respect to the compact open topology.
	\item[(A3)] $Q:\, \mathcal{C}_{\beta} \to \mathcal{C}_{\beta}$ is monotone
	(order preserving) in the sense that $Q[u] \ge
	Q[w]$ whenever $u \ge w$.
	\item[(A4)] $Q$
	admits two $L$-periodic fixed points $0$ and $\beta$ in $\mathcal{C}_+$, and for any $z\in \mathcal{C}^{per}_{\beta}$ with $0\ll z\leqslant \beta$, we have $\lim\limits_{n\rightarrow\infty }Q^n[z](x)=\beta(x)$ uniformly for $x\in \mathbb{R}$.
	\item[(A5)] $Q[\mathcal{C}_{\beta}]$ is precompact in $\mathcal{C}_{\beta}$ with respect to the compact open topology.
\end{itemize}

Define an operator $Q=(Q_1,Q_2)$ on $\mathcal{C}$ by
\begin{eqnarray*}
	& &Q_1[u,v](x)=
	\int_{\mathbb{R}}\frac{r_1(y)u(y)}{1+b_1(y)(u(y)+a_1(y)(q^*(y)-v(y))}k_1(x,y)\di y, \\ 
	& &Q_2[u,v](x)=
	\int_{\mathbb{R}}\frac{r_2(y)}{1+b_2(y)q^*(y)}\cdot \frac{b_2(y)a_2(y)q^*(y)u(y)+v(y)}{1+b_2(y)(q^*(y)+a_2(y)u(y)-v(y))}k_2(x,y)\di y,  \nonumber
\end{eqnarray*}
where $U:=(u,v)\in \mathcal{C}$.

\begin{proposition}
	Assume that (K1)--(K2), and (H1)--(H3) hold. Then $Q$ satisfies the assumptions (A1)--(A5).
\end{proposition}

\begin{proof}
	According to the assumptions (K1)-(K2), it then easily follows that $Q$ is a monotone semiflow on $\mathcal{C}_{\beta}$. Note that if $U_n(x,\phi)=(u_n(x,\phi),v_n(x,\phi))$ is a solution of \eqref{NModel} with $(u_0(\cdot),v_0(\cdot))=(\phi_1,\phi_2):=\phi$, then so is $(u_n(x-a,\phi),v_n(x-a,\phi)),\ \forall a\in L\mathbb{Z}$. This implies that (A1) holds. By Theorem \ref{VLEQ}, it follows that $(A4)$ holds for $Q$. It remains to prove (A2) and (A5). 
	
	We take $Q_1$ as an example, since similar results hold for $Q_2$. Define an operator
	$$G_1(U)(x)=\int_{\mathbb{R}}(\|u\|+\|v\|)\cdot k_1(x,y)\di y,$$
	and 
	$$H_1(U)(x)=\frac{r_1(x)u(x)}{1+b_1(x)(u(x)+a_1(x)(q^*(x)-v(x))}.$$
	It can be verify that $\|G_1\|<+\infty$. For any $\varepsilon>0$, there exist an $\delta_1(\varepsilon)>0$ such that if $U_1=(u^{(1)},v^{(1)}), U_2=(u^{(2)},v^{(2)})\in \mathcal{C}_{\beta}$, with $\|U_1(x)-U_2(x)\|<\delta_1$, we have 
	\begin{align*}
	\|Q_1(U_1)-Q_1(U_2)\|&=\|\int_{\mathbb{R}}(H_1(U_1)(y)-H_1(U_2)(y))\cdot k_1(x,y)\di y\|,\\
	&\leqslant \|\int_{\mathbb{R}}[H_{1u}(\xi)(u^{(1)}-u^{(2)})(y)+H_{1v}(\xi)(v^{(1)}-v^{(2)})(y)]\cdot k_1(x,y)\di y\|\\
	&\leqslant \max\{\|H_{1u}\|,\|H_{1v}\|\}\cdot\|G_1(U_1-U_2)\|<\varepsilon,
	\end{align*}
	which implies that (A2) holds. 
	
	Regarding (A5), it is easy to check $Q_1$ is uniformly bounded. For the above $\varepsilon>0$, there exist an $\delta_2(\varepsilon)>0$ such that $\forall x_1, \ x_2 \in \mathbb{R}$ or any compact interval in $\mathbb{R}$ with $|x_1-x_2|<\delta_2$, since $k_1$ is $L_1$-contiuous, then we have
	$$|Q_1(U)(x_1)-Q_1(U)(x_2))|\leqslant \text{max}\{r_1(x)\}\|U\|\Big|\int_{\mathbb{R}}\Big(k_1(x_1,y)-k_1(x_2,y)\Big)\di y\Big|<\varepsilon,$$
	which implies that $Q_1$ is equicontinous. By the Arzel\`{a}-Ascoli theorem, it follows that $Q_1$ is compact.
	\end{proof}

Now we introduce a family of operators $\{\hat Q\}$ on $\mathcal{X}_\beta$:
\begin{equation}\label{def_Q}
\hat Q[v](s)(\theta):=Q[v_s](\theta),\quad \forall v\in\mathcal{X}_\beta,\ s\in\mathbb{R},\ \theta\in[0,L],
\end{equation}
where $v_s\in\mathcal{C}$ is defined by  
$$v_s(x)=v(s+n_x)(\theta_x),\quad \forall x=n_x+\theta_x\in\mathbb{R},\ n_x=L\left[\frac{x}{L}\right],\ \theta_x\in[0,L).$$
Let $\omega\in \overline{X}_\beta $ with $0\ll\omega\ll \beta$. Choose $\phi \in\mathcal{X}_\beta$ such that the following properties hold:
\begin{enumerate}
	\item[(C1)]$\phi(s)$ is nonincreasing in $s$;
	\item[(C2)] $\phi(s)\equiv0$ for all $s\ge0$;
	\item[(C3)]$\phi(-\infty)=\omega$. 
\end{enumerate}
Let $c$ be a given real number. According to \cite{Wein82}, we define an operator $R_{c}$ by
\[R_c[a](s):=\max\{\phi(s),T_{-c}\hat Q[a](s)\},\]
and a sequence of functions $a_n(c;s)$ by the recursion:
$$a_0(c;s)=\phi(s),\quad a_{n+1}(c;s)=R_{c}[a_n(c;\cdot)](s),$$
where $T_{-c}$ is a translation operator defined by
$T_{-c}[u](x)=u(x+c)$. 
As an consequence of arguments similar to those in \cite[Lemmas 3.1--3.3]{FZ}, we have the following observation.
\begin{lemma}
	The following statements are valid:
	\begin{enumerate}
		\item[(1)]For each $s\in \mathbb{R}$, $a_n(c,s)$ converges to $a(c;s)$ in $X$, where $a(c;s)$ is nonincreasing in both $c$ and $s$, and $a(c;\cdot)\in\mathcal{X}_{\beta}$. 
		\item[(2)]$a(c,-\infty)= \beta$, and $a(c,+\infty)$ exists in $X$ and is a fixed point of $\tilde{P}$. 
	\end{enumerate}
\end{lemma} 
According to \cite{FZ,WLL2002}, we define two numbers 
\begin{eqnarray}\label{defc}
c^*_+=\sup\{c:a(c,+\infty)=\beta\},\quad \overline c_+=\sup\{c:a(c,+\infty)>0\}.
\end{eqnarray}
Clearly, $c^*_+\le \overline{c}_+$ due to the monotonicity of $a(c;\cdot)$ with respect to $c$.
The following two results come from \cite{FZ}.\\

\noindent{\bf Theorem A.} \cite[Theorem 3.8]{FZ} 
	{\it Let $Q$ be a continuous-time semifow on $\mathcal{C}_\beta$ with $Q[0]=0$, $Q[\beta]=\beta$, and   $\hat{Q}$ be defined as in \eqref{def_Q}. Suppose that $Q$ satisfies (A1)--(A5).
	Let $c^*_+$ and $\overline{c}_+$  be defined as in \eqref{defc}. Then the following statements are valid:
	\begin{enumerate}
		\item[(1)] For any $c\ge c^*_+$, there is an $L$-periodic rightward traveling $W(x-cn,x)$ connecting $\beta$ to some equilibrium $\beta_1\in C^{per}_\beta\backslash\{\beta\}$ with $W(\xi,x )$ be continuous and nonincreasing in $\xi\in \mathbb{R}$.
		\item[(2)]If, in addition, $0$ is an isolated equilibrium of $Q$ in $\mathcal{C}^{per}_\beta$, then for any $c\ge\overline{c}_+$ either of the following holds true:\begin{enumerate}
			\item[(i)] there exists an $L$-periodic rightward traveling $W(x-cn,x)$ connecting $\beta$ to $0$ with $W(\xi,x )$ be continuous and nonincreasing in $\xi\in \mathbb{R}$
			\item[(ii)]$Q$ has two ordered equilibria $\alpha_1$,$\alpha_2 \in C^{per}_\beta\backslash\{0,\beta\}$ such that there exist an $L$-periodic traveling wave $W_1(x-cn,x)$ connecting $\alpha_1$ and $0$ and an $L$-periodic traveling wave $W_2(x-cn,x)$ connecting $\beta$ and $\alpha_2$ with $W_i(\xi,x ), i=1,2$ be continuous and nonincreasing in $\xi\in \mathbb{R}$
		\end{enumerate}
		\item[(3)] For any $c< c^*_+$, there is no $L$-periodic traveling wave connecting $\beta$, and for any $c<\overline{c}_+$, there is no $L$-periodic traveling wave connecting $\beta$ to $0$.  
	\end{enumerate}}

\noindent{\bf Theorem B.} \cite[Remark 3.7]{FZ} \
{\it Let $Q$ be a continuous-time semifow on $\mathcal{C}_\beta$ with $Q[0]=0,Q[\beta]=\beta$ and  $\hat{Q}$ be correspondingly defined as in \eqref{def_Q}. Suppose that $Q$ satisfies (A1)--(A5).
	Let $c^*_+$ and $\overline{c}_+$ be defined as in \eqref{defc}. Then the following statements are valid:	
	\begin{enumerate}
		\item[(i)]If $\phi\in\mathcal{C}_{\beta}$, $0\le \phi\le \omega\ll \beta$ for some $\omega\in \mathcal{C}^{per}_{\beta}$, and $\phi(x)=0, \forall x\ge H$, for some $H\in \mathbb{R}$, then $\lim_{n\rightarrow\infty,x\ge cn}Q(\phi)=0$ for any $c>\overline{c}_+$.
		\item[(ii)]If $\phi\in\mathcal{C}_{\beta}$ and $\phi(x)\ge \sigma$, $\forall x\le K$, for some $\sigma\gg0$ and $K\in\mathbb{R}$, then $\lim_{n\rightarrow\infty,x\le cn}(Q(\phi)(x)-\beta(x))=0$ for any $c<c^*_+$.
	\end{enumerate}}

In order to show that $\overline{c}_+$ is the minimal wave speed for $L$-periodic traveling waves of system \eqref{NModel} connecting $\beta$ to $0$, we need the following assumption:
\begin{enumerate}
	\item[(H4)]$c^*_{1+}+c^*_{2-}>0$, where $c^*_{1+}$ and $c^*_{2-}$ are the rightward and leftward spreading speed for \eqref{u1} and \eqref{u2}, respectively.
\end{enumerate}

%

\begin{theorem}\label{MIN}
	Assume that (K1)--(K2), and (H1)--(H4) hold. Then for any $c\ge\overline{c}_+$, system \eqref{NModel} admits an L-periodic traveling wave $(U(x-cn,x),V(x-cn,x))$ connecting $\beta$ to $0$, with wave profile components $U(\xi,x)$ and $V(\xi,x)$ being continuous and non-increasing in $\xi$, and for any $c<\overline{c}_+$, there is no such traveling wave connecting $\beta$ to $0$.
\end{theorem} 
\begin{proof}
	By Theorem A (2) and (3), it suffices to exclude the second case in Theorem A (2). Suppose, by contradiction, the statement in Theorem A (2(ii)) is valid for some $c\ge\overline{c}_+$. Since system \eqref{NModel} has exactly three $L$-periodic nonnegative steady states and $\hat E_0=(0,q^*(x))$ is the only intermediate equilibrium between $\hat E_1=\beta$ and $\hat E_2=0$, we have $\alpha_1=\alpha_2=\hat E_0$. Hence, by restricting  system \eqref{NModel} on the order interval $[\hat{E}_0,\hat{E}_1]$ and $[\hat{E}_2,\hat{E}_0]$, respectively, we find that one scalar equation
	\begin{equation}\label{u1}
	u_{n+1}(x)= \int_{\mathbb{R}}\frac{r_1(y)}{1+b_1(y)u_n(y)}u_n(y)k_1(x,y)\di y,
	\end{equation}
	admits an $L$-periodic traveling wave $U(x-cn,x)$ connecting $p^*(x)$ to $0$ with $U(\xi,x)$ being continuous and nonincreasing in $\xi$, and the other scalar equation 
	\begin{equation}
	v_{n+1}(x)=
	\int_{\mathbb{R}}\frac{r_2(y)}{1+b_2(y)q^*(y)}\cdot \frac{v_n(y)}{1+b_2(y)(q^*(y)-v_n(y))}k_2(x,y)\di y,
	\end{equation}
	also admits an $L$-periodic traveling wave $V(x-cn,x)$ connecting $q^*(x)$
	to $0$ with $V(\xi,x)$ being continuous and nonincreasing in $\xi$.
	
	Let $W(x-cn,x)=q^*(x)-V(x-cn,x)$. Then $W(x-cn,x)$ is an $L$-periodic traveling wave connecting $0$ to $q^*(x)$  of the following scalar equation with $W(\xi,x)$ being continuous and nondecreasing in $\xi$
	\begin{equation}\label{u2}
	w_{n+1}(x)= \int_{\mathbb{R}}\frac{r_2(y)}{1+b_2(y)w_n(y)}w_n(y)k_2(x,y)\di y.
	\end{equation}
	Note that $W(x-cn,x)$ is an $L$-periodic leftward traveling wave connecting $0$ to $q^*$ with wave speed $-c$, and that systems \eqref{u1} and \eqref{u2} admit rightward spreading speed $c^*_{1+}$ and leftward spreading speed $c^*_{2-}$, respectively, which are also the rightward and the leftward minimal wave speeds (see, e.g., \cite[Theorems 5.2 and 5.3]{Liang2}). It then follows that $c\ge c^*_{1+}$ and  $-c\ge c^*_{2-}$. This implies that $c^*_{1+}+c^*_{2-}\leqslant 0$, a contradiction.
\end{proof}
Let $\lambda_2(\mu)$ be the principle eigenvalue of the elliptic eigenvalue problem: 
\begin{small}\begin{align}\label{eep2}
	&\lambda \psi(x)= \int_{\mathbb{R}}\frac{r_2(y)}{1+b_2(y)q^*(y)}e^{-\mu(x-y)}\psi(y)k_2(x,y)\di y,\nonumber\\
	&\psi(x+L)=\psi(x),\quad x\in\mathbb{R}.\end{align} \end{small} 
In order to prove that system \eqref{NModel} admits a single rightward spreading speed, we impose the following assumption:
\begin{enumerate}
	
	\item[(H5)] $\limsup\limits_{\mu\to 0^+}\displaystyle\frac{\ln\lambda _2(\mu)}{\mu}\leqslant c_{1+}^*$, where $c_{1+}^*$ is the rightward spreading speed of \eqref{u1}.
\end{enumerate} 
\begin{theorem}\label{Qspreading}
	Assume that (K1)--(K2), and (H1)--(H5) hold. Then the following statements are valid for system \eqref{NModel}:
	\begin{enumerate}
		\item[(i)]If $\phi\in\mathcal{C}_{\beta}$, $0\leqslant \phi\leqslant \omega\ll \beta$ for some $\omega\in \mathcal{C}^{per}_{\beta}$, and $\phi(x)=0, \forall x\ge H$, for some $H\in \mathbb{R}$, then $\lim\limits_{n\rightarrow\infty,x\ge cn}(u_n(x,\phi),v_n(x,\phi))=(0,0)$ for any $c>\overline{c}_+$.
		\item[(ii)]If $\phi\in\mathcal{C}_{\beta}$ and $\phi(x)\ge \sigma$, $\forall x\leqslant K$, for some $\sigma\in \mathbb{R}^2$ with $\sigma\gg0$ and $K\in\mathbb{R}$, then $\lim\limits_{n\rightarrow\infty,x\leqslant cn}((u_n(x,\phi),v_n(x,\phi))-\beta(x))=0$ for any $c<\overline{c}_+$.
	\end{enumerate}
\end{theorem}
\begin{proof}
	By Theorem B, it suffices to show $\overline{c}_+=c^*_+$. If this is not valid, then the definition of $\overline{c}_+$ and $c^*_+$ implies that $\overline{c}_+>c^*_+$. By Theorem A (1) and (3), it follows that system \eqref{NModel} admits an $L$-periodic traveling wave $(U(x-c^*_+n,x),V(x-c^*_+n,x))$ connecting $(p^*(x),q^*(x))$ to $(0,q^*(x))$ with $U(\xi,x)$ and $V(\xi,x)$ being continuous and nonincreasing in $\xi$. Therefore, $V\equiv q^*(x)$, and $U_1(x-c^*_+n,x)$ is an $L$-periodic traveling wave connecting $p^*(x)$ to $0$. This implies that $c^*_+\ge c^*_{1+}$ where $c^*_{1+}$ is the rightward spreading of \eqref{u1}. By \cite[(2.7)]{Wein2008}, it follows that $c^*_{1+}=\inf_{\mu>0}\frac{\ln\lambda_1(\mu)}{\mu}$, where $\lambda_1(\mu)$ is the principal eigenvalue of the following eigenvalue problem:
	\begin{eqnarray}\label{eep}
	& & \lambda \psi(x)= \int_{\mathbb{R}}r_1(y)e^{-\mu(y-x)}\psi(y)k_1(x,y)\di y, \nonumber\\
	& &\psi(x+L)=\psi(x),\quad x\in\mathbb{R}.
	\end{eqnarray}
	For any given $c_1\in(c^*_+,\overline{c}_+)$, there exists $\mu_1>0$ such that $c_1=\frac{\ln\lambda_1(\mu_1)}{\mu_1}$. Let $\phi^*_1(x)$ be the $L$-periodic positive
	eigenfunction associated with the principal eigenvalue $\lambda_1(\mu_1)$ of \eqref{eep}. It then easily follows that 
	$$u_n(x):=e^{-\mu_1(x-c_1n)}\phi^*_1(x)=e^{-\mu_1x}\phi^*_1(x)[\lambda_1(\mu_1)]^n,\quad n\ge0,\ x\in\mathbb{R},$$
	is a solution of the linear equation
	$$
	u_{n+1}(x)= \int_{\mathbb{R}}r_1(y)u_n(y)k_1(x,y)\di y.
	$$
	Since $c^*_1<c_1$ and (H5) holds,  we can choose a
	small number $\mu_2\in (0, \mu_1)$ such that  $c_2:=\frac{\ln\lambda_2(\mu_2)}{\mu_2}< c_1$.
	Let $\phi_2^*(x)$ be the positive
	eigenfunction associated with the principal eigenvalue $\lambda_2(\mu_2)$ of \eqref{eep2}.
	It is easy to see that
	$$
	v_n(x):=e^{-\mu_2(x-c_2n)}\phi^*_2(x)=e^{-\mu_2x}\phi^*_2(x)[\lambda_2(\mu_2)]^n
	$$
	is a solution of the linear equation
	\begin{equation}\label{Eq1}
	v_{n+1}(x)= \int_{\mathbb{R}}\frac{r_2(y)}{1+b_2(y)q^*(y)}v_n(y)k_2(x,y)\di y.
	\end{equation}
	Since $c_1>c_2$, it follows that the function
	$$
	\tilde{v}_n(x):=e^{-\mu_2(x-c_1n)}\phi^*_2(x)=
	e^{\mu_2(c_1-c_2)n}v_n(x),\quad n\ge0,\ x\in\mathbb{R},
	$$
	satisfies
	\begin{equation}\label{ineq}
	\tilde{v}_{n+1}(x)\ge \int_{\mathbb{R}}\frac{r_2(y)}{1+b_2(y)q^*(y)}\tilde{v}_n(y)k_2(x,y)\di y.
	\end{equation}
	
	Define the following two functions: 
	\begin{equation}
	\overline{u}_n(x):=\min\{h_1e^{-\mu_1(x-c_1n)}\phi^*_1(x),p^*(x)\},\quad n\ge0,\ x\in \mathbb{R},
	\end{equation} and
	\begin{equation}
	\overline{v}_n(x):=\min\{h_2e^{-\mu_2(x-c_1n)}\phi^*_2(x),q^*(x)\},\quad n\ge0,\ x\in\mathbb{R},
	\end{equation}
	where
	$$h_2:= \max_{x\in[0,L]}\frac{q^*(x)}{\phi^*_2(x)}>0,\quad h_1:=\min_{x\in[0,L]}\frac{h_1\phi^*_2(x)}{b_{2}(x)\phi^*_1(x)}>0.$$
	Now we want to verify that $(\overline{u}_n,\overline{u}_n)$ is an upper solution for system \eqref{NModel}. For all $x-c_1 n>\frac{1}{\mu_1}\ln\frac{h_1\phi_1^*(x)}{p^*(x)}$, we have $\overline{u}_n(x)=h_1e^{-\mu_1(x-c_1n)}\phi^*_1(x)$, and therefore,
	\begin{eqnarray*}
		& &\quad\overline{u}_{n+1}(x)-Q_1[\overline{u}_{n},\overline{v}_{n}](x)\\
		& &
		=\int_{\mathbb{R}}\frac{r_1(y)b_1(y)\overline{u}_n(y)[\overline{u}_n(y)+a_1(y)(q^*(y)-\overline{v}_n(y))]}{1+b_1(y)(\overline{u}_n(y)+a_1(y)(q^*(y)-\overline{v}_n(y))}k_1(x,y)\di y\geqslant 0.
	\end{eqnarray*}
	For all $x-c_1n<\frac{1}{\mu_1}\frac{h_1\phi_1^*(x)}{p^*(x)}$, we obtain $\overline{u}_n(x)=p^*(x)$, and hence,
	\begin{eqnarray*}
		& &\quad\overline{u}_{n+1}(x)-Q_1[\overline{u}_{n},\overline{v}_{n}](x)\\
		& &=\int_{\mathbb{R}}\frac{r_1(y)b_1(y)a_1(y)p^*(y)[q^*(y)-\overline{v}_n(y)]}{[1+b_1(y)p^*(y)][1+b_1(y)(p^*(y)+a_1(y)(q^*(y)-\overline{v}_n(y))]}k_1(x,y)\di yy\geqslant 0.
	\end{eqnarray*}
	On the other hand, for all $x-c_1n\!>\!\frac{1}{\mu_2}\ln\frac{h_2\phi_2^*(x)}{q^*(x)}\!\geqslant \!0$, it follows that  $$\overline{v}_n(x)=h_2e^{-\mu_2(x-c_1n)}\phi^*_2(x),$$ which satisfies inequality  \eqref{ineq}. Note that $$\overline{u}_n(x)\leqslant h_1e^{-\mu_1(x-c_1n)}\phi^*_1(x), \quad \forall t\ge0,\ x\in\mathbb{R},$$ and $\mu_2\in(0,\mu_1)$, we have 
	\begin{eqnarray*}
		& &\quad\overline{v}_{n+1}(x)-Q_2[\overline{u}_{n},\overline{v}_{n}](x)\\
		& &
		= \int_{\mathbb{R}}\frac{r_2(y)b_2(y)}{1+b_2(y)q^*(y)}\cdot \frac{(q^*(y)-\overline{v}_n(y)(\overline{v}_n(y)-b_2(y)\overline{u}_n(y))}{1+b_2(y)(q^*(y)+a_2(y)\overline{u}_n(y)-\overline{v}_n(y))}k_2(x,y)\di yy\geqslant 0,
	\end{eqnarray*}
	where $\overline{v}_n-a_2\overline{u}_n=e^{-\mu_1(x-c_1n)}(h_2\phi^*_2-b_2h_1\phi^*_1)\ge0.$
	
	For all $x-c_1n<\frac{1}{\mu_2}\ln\frac{h_2\phi_2^*(x)}{q^*(x)}$, we have $\overline{v}_n(x)=q^*(x)$. Therefore,
	\begin{eqnarray*}
		& &\quad\overline{v}_{n+1}(x)-Q_2[\overline{u}_{n},\overline{v}_{n}](x)\\
		& &
		= \int_{\mathbb{R}}\frac{r_2(y)q^*(y)}{1+b_2(y)q^*(y)}\big[1- \frac{1+b_2(y)a_2(y)\overline{u}_n(y)}{1+b_2(y)a_2(y)\overline{u}_n(y)}\big]k_2(x,y)\di y=0,
	\end{eqnarray*}
	It then follows that $\overline{U}_n:=(\overline{u}_n,\overline{v}_n)$ is a continuous upper  solution of system \eqref{NModel}.
	\smallskip
	
	Let $\phi\in \mathcal{C}_\beta$  with $\phi(x)\ge \sigma$, $\forall x\leqslant K$ and $\phi(x)=0$, $\forall x\ge H$, for some $\sigma\in \mathbb{R}^2$ with $\sigma\gg0$ and $K,H\in\mathbb{R}$. By the arguments in \cite[Lemma 2.2]{WLL2002} and \cite[Theorem 5.4]{Yu}, it follows that for any $c<\overline{c}_+$, there exists $\delta(c)>0$ such that 
	\begin{equation}\label{ineq2}{\lim\inf}_{n\rightarrow\infty, x\leqslant cn}|U_n(x,\phi)|\ge\delta(c)>0.\end{equation} Moreover, there exists a sufficiently large positive constant $A\in L\mathbb{Z}$ such that 
	$$\phi(x)\leqslant\overline{U}_0(x-A) :=\psi(x),\quad \forall x\in \mathbb{R}.$$
	By the translation invariance of $Q$, it follows that $\overline{U}_n(x-A)=(\overline{u}_n(x-A),\overline{v}_n(x-A))$ is still an upper solution of system \eqref{NModel}, and hence for $U_n$, we have
	\begin{equation}\label{ineq3}0\leqslant U_n(x,\phi)\leqslant U_n(x,\psi)=\overline{U}_n(x-A),\quad \forall x\in\mathbb{R},\  n\ge0.\end{equation}
	Fix a number $\hat c\in(c_1,\overline{c}_+)$. Letting $x=\hat cn$ and $n\rightarrow\infty$ in \eqref{ineq3}, together with \eqref{ineq2}, we have 
	$$0<\delta(\hat c)\leqslant\liminf_{n\rightarrow\infty}|U_n(\hat{c}n,\phi)|\leqslant\lim_{n\rightarrow\infty}|\overline{U}_n(\hat cn-A)|=0,$$  
	which is a contradiction. Thus, we have $c^*_+=\overline{c}_+$.
\end{proof}


To finish this section, we present some results on the principle eigenvalue problem.
\begin{proposition}\label{rem}
	Let $\lambda_m(\mu)(\mu \in [0,\mu^*))$ be the principle eigenvalue of the following  eigenvalue problem:
	\begin{eqnarray}\label{geep}
	& & \lambda \psi=\int_{\mathbb{R}}m(y)e^{-\mu(y-x)}\psi(y)k(x,y)\di y,\\
	& &\psi(x+L)=\psi(x),\quad x\in\mathbb{R}. \nonumber
	\end{eqnarray}
	Then the following statements are valid:
	\begin{enumerate}
		\item[(i)]If $m_1(x)\geqslant m_2(x)>0$ with $m_1(x)\not\equiv m_2(x),\forall x\in\mathbb{R}$, then $\lambda_{m_1}(\mu)>\lambda_{m_2}(\mu)$.
		\item[(ii)] $\ln\lambda_m(\mu)$ is a convex function of $\mu$ on $\mathbb{R}$.
		\item[(iii)]If $k(x,y)=k(y,x)$, then $\lambda_m(\mu)=\lambda_m(-\mu)$.
	\end{enumerate}
\end{proposition}
\begin{proof}
	We use the arguments similar to those in \cite[Lemma 15.5]{Hess}  
	to prove that (a) holds. First we define
	\[\check{L}_m[\psi](x)=\int_{\mathbb{R}}m(y)e^{-\mu(y-x)}\psi(y)k(x,y)\di y.\]
	Let $\lambda_{m_1}(\mu)$, $\lambda_{m_2}(\mu)$ be principal eigenvalues with $m_1(x)\geqslant m_2(x)>0$, and $m_1(x)\not\equiv m_2(x)$. Suppose by contradiction, $\lambda_{m_1}(\mu)\leqslant\lambda_{m_2}(\mu)$. Let $\psi_1$, $\psi_2$ be associated eigenfunctions and chosen in a way that $0<\psi_2\ll \psi_1$. Then
	\[\check{L}_{m_2}(\psi_1-\psi_2)<\check{L}_{m_1}\psi_1-\check{L}_{m_2}\psi_2=\lambda_{m_1}\psi_1-\lambda_{m_2}\psi_2\leqslant \lambda_{m_1}(\psi_1-\psi_2).\]
	It follows that $\lambda_{m_1}(\psi_1-\psi_2)-\check{L}_{m_2}(\psi_1-\psi_2)=h\gg 0$, and hence, $\psi_1-\psi_2$ is a positive root of $\lambda_{m_1}\psi-\check{L}_{m_2}\psi=0$, which is a contradiction to \cite[Theorem 7.2]{Hess}, which states the above euqation has no positive solution if $\lambda_{m_1}(\mu)\leqslant\lambda_{m_2}(\mu)$.
	
	(b) follows from the same argument as in \cite[Lemma 3.7]{Liang}. (c) can be proved by the arguments similiar to those in \cite[Theorem 2.3]{Ding13}.
\end{proof}

\section{Linear determinacy of spreading speed}
In this section, we establish a set of sufficient conditions for the rightward spreading speed to be determined by the linearization of system \eqref{NModel} at $\hat E_1=(0,0)$, which is
\begin{align}\label{linear}
&u_{n+1}(x)=
\int_{\mathbb{R}}\frac{r_1(y)u_n(y)}{1+b_1(y)a_1(y)q^*(y)}k_1(x,y)\di y, \\
&v_{n+1}(x)=
\int_{\mathbb{R}}\frac{r_2(y)}{1+b_2(y)q^*(y)}\cdot \frac{b_2(y)a_2(y)q^*(y)u_n(y)+v_n(y)}{1+b_2(y)q^*(y)}k_2(x,y)\di y,\quad n>0,\  x\in\mathbb{R}\nonumber.
\end{align} 

Under (H2) the following scalar equation 
\begin{equation}\label{c0eq}
u_{n+1}(x)=
\int_{\mathbb{R}}\frac{r_1(y)u_n(y)}{1+b_1(y)(u_n(y)+a_1(y)q^*(y))}k_1(x,y)\di y,\quad n>0, x\in\mathbb{R}, \\
\end{equation}
admits a rightward spreading speed (also minimal rightward wave speed) $c^0_+=\inf\limits_{\mu>0}\frac{\ln\lambda_0(\mu)}{\mu}$ (see, e.g., \cite{Wein2008}), where $\lambda_0(\mu)$ is the principle eigenvalue of the following eigenvalue problem:
\begin{small}
	\begin{align}\label{eep0}
	&  \lambda \psi(x)= \int_{\mathbb{R}}\frac{r_1(y)}{1+b_1(y)a_1(y)q^*(y)}e^{-\mu(y-x)}\psi(y)k_1(x,y)\di y,\nonumber\\
	& \psi(x+L)=\psi(x),\quad x\in\mathbb{R}.
	\end{align}
\end{small}
The subsequent result shows that $c^0_+$ is a lower bound of the slowest spreading $c^*_+$ of system \eqref{NModel}.
\begin{proposition}\label{lb}Let (K1)--(K2), and (H1)--(H3) hold. Then $c^*_+\ge c^0_+$.
\end{proposition}
\begin{proof}In the case that $\overline{c}_+>c^*_+$, by the same arguments as in Theorem \ref{Qspreading}, we see that $c^*_+\ge c^*_{1+}$, where $c^*_{1+}$ is the rightward spreading speed of \eqref{u1}. Since $r_1(x)>\displaystyle\frac{r_1(x)}{1+b_1(x)a_1(x)q^*(x)}, \forall x\in \mathbb{R}$, by Proposition \ref{rem} (i), we have $\lambda_1(\mu)>\lambda_0(\mu), \forall \mu\ge0$,  where  $\lambda_1(\mu)$ is the principal eigenvalue of \eqref{eep}. Thus, we have $c^*_+\ge c^*_{1+}>c^0_+$.
	
	In the case that $\overline{c}_+=c^*_+$, let $(u_n(\cdot,\phi),v_n(\cdot,\phi))$ be the solution of system \eqref{NModel} with $\phi=(\phi_1,\phi_2)\in \mathcal{C}_\beta$. Then the positivity of the solution implies that
	$$u_{n+1}(x)\geqslant \int_{\mathbb{R}}\frac{r_1(y)}{1+b_1(y)a_1(y)q^*(y)}u_n(y)k_1(x,y)\di y,$$
	Let $w_n(x,\phi_1)$ be the unique solution of \eqref{c0eq} with $w_0(\cdot)=\phi_1$. Then the comparison principle yields that
	\begin{equation}\label{ineq1}u_n(x,\phi)\geqslant w_n(x,\phi_1), \quad \forall t\ge0,\  x\in \mathbb{R}.\end{equation}
	Since $\lambda(k_1,\displaystyle\frac{r_1}{1+b_1a_1q^*})>1$, Proposition \ref{VLexistence} implies that there exists a unique positive $L$-periodic steady state $w^*(x)$ of \eqref{c0eq}. Let $\phi^0=(\phi_1^0,\phi_2^0)\in \mathcal{C}_\beta$ be choosen as in Theorem \ref{Qspreading} (i) and (ii) such that $\phi_1^0\leqslant w^*$. 
	%
	%
	%
	Suppose, by the contradiction, that $c_+^*<c^0_+$. We choose some $\hat c\in(\overline{c}_+^*,c^0_+)$. Then
	Theorem \ref{Qspreading} implies $\lim_{n\rightarrow\infty,x\ge \hat cn}u_n(x,\phi^0)=0$. By Theorem B as applied to system \eqref{c0eq}, we have $\lim_{n\rightarrow\infty,x\leqslant \hat cn}(w_n(x,\phi^0_1)-w^*(x))=0$.
	However, letting $x=\hat c n$ in \eqref{ineq1},  we get  $\lim_{n\rightarrow\infty,x= \hat cn}w_n(x,\phi^0_1)=0$, which is a contradiction.
\end{proof}

For any given $\mu\in \mathbb{R}$, letting $U_n(x)=e^{-\mu x}\phi(x)[\lambda(\mu)]^n$ in \eqref{linear},
we obtain the following periodic eigenvalue problem:
\begin{align}\label{Lpep}
&\lambda \phi_1= \int_{\mathbb{R}}\frac{r_1(y)}{1+b_1(y)a_1(y)q^*(y)}e^{-\mu(y-x)}\phi_1(y)k_1(x,y)\di y,\nonumber \\
&\lambda \phi_2=\int_{\mathbb{R}}\frac{r_2(y)}{1+b_2(y)q^*(y)}\cdot \frac{b_2(y)a_2(y)q^*(y)\phi_1(y)+\phi_2(y)}{1+b_2(y)q^*(y)}e^{-\mu(y-x)}k_2(x,y)\di y,\\
&\phi_i(x)=\phi_i(x+L), \quad \forall x\in\mathbb{R},\ i=1,2.\nonumber
\end{align}
Let $\overline{\lambda}(\mu)$ be the principal eigenvalue of the following  periodic eigenvalue problem:
\begin{align}\label{Lpep2}
&\lambda \psi=\int_{\mathbb{R}}\frac{r_2(y)}{1+b_2(y)q^*(y)}\cdot \frac{\psi(y)}{1+b_2(y)q^*(y)}e^{-\mu(y-x)}k_2(x,y)\di y,\\
&\psi(x)=\psi(x+L), \quad x\in\mathbb{R}.\nonumber
\end{align}
Then there exists $\mu_0>0$ such that $c^0_+=\frac{\ln\lambda_0(\mu_0)}{\mu_0}$. Now we make the following assumption:
\begin{enumerate}	
	\item[(D1)] $\lambda_0 (\mu_0)>\overline{\lambda}(\mu_0)$.
\end{enumerate} 
\begin{proposition}\label{ef}
	Let (A1)--(A2), (H1)--(H3) and (D1) hold. Then the periodic eigenvalue problem \eqref{Lpep} with $\mu=\mu_0$ has a simple eigenvalue $\lambda_0(\mu_0)$ associated with a positive $L$-periodic eigenfunction $\phi^*=(\phi_1^*,\phi_2^*)$.  
\end{proposition}
\begin{proof}
	Clearly, there exists an $L$-periodic eigenfunction $\phi_1^*\gg0$ associated with the principle eigenvalue $\lambda_0(\mu_0)$ of \eqref{c0eq}, that is,
	$$\lambda_0(\mu_0)\phi_1^*= \int_{\mathbb{R}}\frac{r_1(y)}{1+b_1(y)a_1(y)q^*(y)}e^{-\mu_0(y-x)}\phi_1^*(y)k_1(x,y)\di y.$$
	Since the first equation of \eqref{Lpep} is decoupled from the second one, it suffices to show that $\lambda_0(\mu_0)$ has a positive eigenfunction $\phi^*=(\phi^*_1,\phi^*_2)$ in \eqref{Lpep}, where $\phi^*_2$ is to be determined. Note that
	\begin{align*}
	\lambda \phi_2&=\int_{\mathbb{R}}\frac{r_2(y)}{1+b_2(y)q^*(y)}\cdot \frac{b_2(y)a_2(y)q^*(y)\phi_1^*(y)+\phi_2(y)}{1+b_2(y)q^*(y)}e^{-\mu(y-x)}k_2(x,y)\di y,\\
	&=\int_{\mathbb{R}}\frac{r_2(y)}{1+b_2(y)q^*(y)}\cdot \frac{b_2(y)a_2(y)q^*(y)\phi_1^*(y)}{1+b_2(y)q^*(y)}e^{-\mu(y-x)}k_2(x,y)\di y\\
	&+\int_{\mathbb{R}}\frac{r_2(y)}{1+b_2(y)q^*(y)}\cdot \frac{\phi_2(y)}{1+b_2(y)q^*(y)}e^{-\mu(y-x)}k_2(x,y)\di y\\
	&:=h+\tilde{L}\phi_2.
	\end{align*}
	It follows that $\lambda_0(\mu_0)\phi_2-\tilde{L}\phi_2=h\gg0$.
	It is easy to verify $\tilde{L}$ is a positive and compact, and hence, $s(L)=\overline{\lambda}(\mu_0)$, where $s(L)$ is the spectral radius of $L$. Since $\lambda_0(\mu_0)>\overline\lambda(\mu_0)=s(L)$, by the Krein-Rutman Theorem \cite{Krein}, there exists a unique $\phi_2^*\gg0$ such that  $\lambda_0(\mu_0)\phi_2^*-\tilde{L}\phi_2^*=h\gg0$. It then follows that $(\phi^*_1,\phi^*_2)$ satisfies \eqref{Lpep} with $\mu=\mu_0$. Since $\lambda_0(\mu_0)$ is a simple eigenvalue for \eqref{c0eq}, we see that so is $\lambda_0(\mu_0)$ for \eqref{Lpep}.
\end{proof}
\quad By virtue of  Proposition \ref{ef}, we easily see that for any given $M>0$, the function \begin{equation}\label{eqU}
S_n(x)=Me^{-\mu_0x}[\lambda_0(\mu_0)]^n\phi^*(x),\quad n\ge0,\ x\in\mathbb{R},
\end{equation} where $S_n(x)=(s_n(x),w_n(x))$, is a positive solution of system \eqref{linear}. In order to obtain an explicit formula for the spreading speeding $\overline{c}_+$, we need the following additional condition:
\begin{enumerate}
	\item[(D2)] $\displaystyle\frac{\phi^*_1(x)}{\phi^*_2(x)}\geqslant \max\left\{a_1(x),\frac{1}{a_{2}(x)}\right\},\quad \forall x\in\mathbb{R}$.
\end{enumerate}

We are now in a position to show that system \eqref{NModel} admits a single rightward spreading speed $\overline{c}_+$, which is linearly determinate.

\begin{theorem}\label{c0}
	Let (K1)--(K2), (H1)--(H3) and (D1)--(D2) hold. Then $\overline{c}_+=c^*_+=c^0_+=\inf_{\mu>0}\frac{\ln\lambda_0(\mu)}{\mu}$.
\end{theorem}
\begin{proof}
	First, we verify that $S_n(x)=(s_n,w_n)$, as defined in \eqref{eqU}, is an upper solution of system \eqref{NModel}. Since $\displaystyle\frac{s_n}{w_n}=\frac{\phi^*_1}{\phi^*_2}$ and (D2) holds , it follows that 
	\begin{eqnarray}
	& &s_{n+1}(x)-\int_{\mathbb{R}}\frac{r_1(y)s_n(y)}{1+b_1(y)(s_n(y)+a_1(y)(q^*(y)-w_n(y))}k_1(x,y)\di y\nonumber\\
	& &=
	\int_{\mathbb{R}}\frac{r_1(y)b_1(y)s_n(y)w_n(y)k_1(x,y)}{[1+b_1(y)a_1(y)q^*(y)][1+b_1(y)(s_n(y)+a_1(y)(q^*(y)-w_n(y))]}\Big(\frac{s_n(y)}{w_n(y)}-a_1(y)\Big)\di y \nonumber\\
	& &=\int_{\mathbb{R}}\frac{r_1(y)b_1(y)s_n(y)w_n(y)k_1(x,y)}{[1+b_1(y)a_1(y)q^*(y)][1+b_1(y)(s_n(y)+a_1(y)(q^*(y)-w_n(y))]}\Big(\frac{\phi_1^*(y)}{\phi_2^*(y)}-a_1(y)\Big)\di y \nonumber\\
	& &\geqslant 0,
	\end{eqnarray}
	and
	\begin{eqnarray}
	& & w_{n+1}(x)-\int_{\mathbb{R}}\frac{r_2(y)}{1+b_2(y)q^*(y)}\cdot\frac{r_2(y)a_2(y)q^*(y)s_n(y)+w_n(y)}{1+r_2(y)(q^*(y)-w_n(y)+a_2(y)s_n(y))}k_2(x,y)\di y\nonumber\\
	& &=
	\int_{\mathbb{R}}\frac{r_2(y)b_2(y)w_n(y)k_2(x,y)}{1+b_2(y)q^*(y)}\cdot\frac{b_2(y)a_2(y)q^*(y)s_n(y)+w_n(y)}{[1+b_2(y)q^*(y)][1+b_2(y)(q^*(y)-w_n(y)+a_2(y)s_n(y))]}\cdot \nonumber\\
	& &\quad\Big(a_2(y)\frac{s_n(y)}{w_n(y)}-1\Big)\di y\nonumber\\
	& &=
	\int_{\mathbb{R}}\frac{r_2(y)b_2(y)w_n(y)k_2(x,y)}{1+b_2(y)q^*(y)}\cdot\frac{b_2(y)a_2(y)q^*(y)s_n(y)+w_n(y)}{[1+b_2(y)q^*(y)][1+b_2(y)(q^*(y)-w_n(y)+a_2(y)s_n(y))]}\cdot \nonumber\\
	& &\quad\Big(a_2(y)\frac{\phi_1^*(y)}{\phi_2^*(y)}-1\Big)\di y\nonumber\\
	& &\geqslant 0,
	\end{eqnarray}
	Thus, $S_n(x)$ is an upper solution of \eqref{NModel}. As we did in the proof of Proposition \ref{lb}, we can choose some $\phi^0\in\mathcal{C}_\beta$ satisfying the conditions in Theorem \ref{Qspreading} (i) and (ii). Then there exists a sufficiently large number $M_0>0$ such that 
	$$0\leqslant \phi^0(x)\leqslant M_0e^{-\mu_0x}\phi^*(x)=S_0(x),\quad \forall x\in\mathbb{R}.$$
	Let $U_n(x)$ be the unique solution of system \eqref{NModel} with $U_0(\cdot)=\phi_0$. Then the comparison principle, together with the fact that $c^0_+\mu_0=\ln\lambda_0(\mu_0)$, gives rise to  
	$$0\!\leqslant\! U_n(x)\!\leqslant\! S_n(x)\!=\!M_0e^{-\mu_0x}\lambda_0(\mu_0)^n\phi^*(x)\!=\!M_0e^{-\mu_0(x-c^0_+)n}\phi^*(x),\quad \forall n\geqslant 0,\ x\in\mathbb{R}.$$
	It follows that for any given $\varepsilon>0$, there holds
	$$0\leqslant U_n(x)\leqslant S_n(x)\leqslant M_0e^{-\mu_0\varepsilon n}\phi^*(x),\quad \forall n\geqslant 0,\ x\geqslant (c^0_++\varepsilon)n,$$
	and hence,
	$$\lim_{n\rightarrow\infty,x\geqslant (c^0_++\varepsilon)n}U_n(x)=0.$$
	By Theorem B (ii), we obtain $c^*_+\leqslant c^0_++\varepsilon$. Letting $\varepsilon\rightarrow 0$, we have $c^*_+\geqslant c^0_+$. In the case that $\overline{c}_+>c^*_+$, the proof of Proposition \ref{lb} shows that $c^*_+>c^0_+$, a contradiction. This shows that $\overline{c}_+=c^*_+=c^0_+$. 
\end{proof}

\section{An application}
In this section, we assume the $k_i(x,y) \ (i=1,2)$ can be written as a function of the dispersal distance, i.e., $k_i(x,y)=k_i(x-y)$, with the following property:
\begin{itemize}
	\item[(K3)] $\displaystyle \int_{-\infty}^{+\infty}k_i(x-y)\di y=1$, and $k_i(-y)=k_i(y)$.  
\end{itemize}

As an application, we consider a patchy lanscape in which both species have the same spatially varying carrying capacity, $C_1(x)=C_2(x)=C(x)$, that is, 
\begin{align*}
	C(x)=\left\{\begin{array}{l}
		C_M, \quad\quad\quad 0\leqslant x<L_1, \\[2mm]
		C_m<C_M,  \ L_1\leqslant x<L,
	\end{array}\right.
\end{align*}
This indicates that Patch 1 is more suitable for both species, compared with Patch 2. The growth rates $r_i$ of $i$-th species ($i=1,2$) are constant, which are environmental homogeneous, and $a_i(x)$ are also piecewise constant functions. Therefore, we are led to the following spatically periodic model with kernels $k_i$ satisfying assumptions (K2)-(K3):
\begin{eqnarray}\label{dhmp}
	& &p_{n+1}(x)=
	\int_{\mathbb{R}}\frac{r_1p_n(y)}{1+b_1(y)(p_n(y)+a_1(y)q_n(y))}k_1(x-y)\di y, \\ 
	& &q_{n+1}(x)=
	\int_{\mathbb{R}}\frac{r_2q_n(y)}{1+b_2(y)(q_n(y)+a_2(y)p_n(y))}k_2(x-y)\di y, \ x\in\mathbb{R}, \nonumber
\end{eqnarray}
where 
$b_i(x)=\displaystyle\frac{r_i-1}{C(y)}.$

We also need the following assumption on system (\ref{dhmp}): 
\begin{enumerate}
	\item[(M)] $a_1^{M}<\displaystyle \frac {C_m}{C_M}$, and $\displaystyle \frac {C_M}{C_m}<a_2^{m}$, where
	$a_1^{M}=\max_{x\in[0,L]}a_1(x),\ a_2^{m}=\min_{x\in[0,L]}a_2(x)$.
\end{enumerate}
\begin{lemma}\label{M}
	Let (M) hold and assume that kernels $k_i$ satisfy (K1)--(K3). Then (H1)--(H5) are valid for system \eqref{dhmp}. 
\end{lemma}

\begin{proof}
	(H1) holds immedately by Proposition \ref{rem} (i) with $m_i(x)=r_i>1$.
	
	Now we verify (H2). Let $(0,q^*)$ be the $L$-periodic semi-trivial steady state of sytem (\ref{dhmp}), which is guaranteed by (H1), and $q_0=\max\limits_{x\in[0,L]}q^*(x)$. By a comparison argument, we have
	\[\int_{\mathbb{R}}\frac{r_2q_0}{1+\frac{r_2-1}{C(y)}q_0}k_2(x-y)\di y\geqslant \int_{\mathbb{R}}\frac{r_2q^*}{1+\frac{r_2-1}{C(y)}q^*}k_2(x-y)\di y=q^*.\]
	It then follows that
	\[\int_{\mathbb{R}}\frac{r_2q_0}{1+\frac{r_2-1}{C(y)}q_0}k_2(x-y)\di y\geqslant q_0,\]
	and hence,
	\[\int_{\mathbb{R}}\frac{r_2q_0}{1+\frac{r_2-1}{C_M}q_0}k_2(x-y)\di y
	\geqslant\int_{\mathbb{R}}\frac{r_2q_0}{1+\frac{r_2-1}{C(y)}q_0}k_2(x-y)\di y\geqslant q_0.\]
	Then we have
	\[\frac{r_2}{1+\frac{r_2-1}{C_M}q_0}\geqslant 1, \ \text{i.e.,}\ q^*\leqslant q_0\leqslant C_M. \]
	Taking $m_1(x)=\displaystyle\frac {r_1}{1+b_1(x)a_1(x)q^*(x)}=\frac{r_1}{1+\frac{r_1-1}{C(y)}a_1(x)q^*(x)}$, we obtain 
	\[\frac{r_1}{1+\frac{r_1-1}{C(y)}a_1(x)q^*(x)}\geqslant \frac{r_1}{1+\frac{r_1-1}{C_m}a_1^Mq^*(x)}\geqslant \frac{r_1}{1+\frac{r_1-1}{C_m}a_1^MC_M}>\frac {r_1}{1+r_1-1}=1.\]
	Thus, $\lambda\Big(k_1, \displaystyle\frac {r_1}{1+b_1a_1q^*}\Big)>1$ due to Proposition \ref{rem} (i).
	
	Next we prove (H3) by a way of contradiction. Suppose $(\tilde{p},\tilde{q})$ is the positive $L$-periodic steady state. We introduce the following system 
	\begin{eqnarray}\label{compari}
	& &p_{n+1}(x)=
	\int_{\mathbb{R}}\frac{r_1p_n(y)}{1+\frac{r_1-1}{C_m}(p_n(y)+a_1^Mq_n(y))}k_1(x-y)\di y, \\ 
	& &q_{n+1}(x)=
	\int_{\mathbb{R}}\frac{r_2q_n(y)}{1+\frac{r_2-1}{C_M}(q_n(y)+a_2^mp_n(y))}k_2(x-y)\di y, \ x\in\mathbb{R}, \nonumber
	\end{eqnarray}
	with $\tilde{p}_0=\min\limits_{x\in[0,L]}\tilde{p}(x)$, $\tilde{q}_0=\min\limits_{x\in[0,L]}\tilde{q}(x)$. Since $(p_n(x), q_n(x))$ satisfies system (\ref{dhmp}), we have
	\begin{eqnarray*}
		& &p_{n+1}(x)\geqslant
		\int_{\mathbb{R}}\frac{r_1p_n(y)}{1+\frac{r_1-1}{C_m}(p_n(y)+a_1^Mq_n(y))}k_1(x-y)\di y, \\ 
		& &q_{n+1}(x)\leqslant
		\int_{\mathbb{R}}\frac{r_2q_n(y)}{1+\frac{r_2-1}{C_M}(q_n(y)+a_2^mp_n(y))}k_2(x-y)\di y, \ x\in\mathbb{R}.
	\end{eqnarray*}
	By the comparison argument, we easily verify that
	\[\int_{\mathbb{R}}\frac{r_1\tilde{p}_0}{1+\frac{r_1-1}{C_m}(\tilde{p}_0+a_1^M\tilde{q}_0)}k_1(x-y)\di y\leqslant \int_{\mathbb{R}}\frac{r_1\tilde{p}(y)}{1+\frac{r_1-1}{C(y)}(\tilde{p}(y)+a_1(y)\tilde{q}(y))}k_1(x-y)\di y=\tilde{p}(x).\]
	It follows that 
	\[\frac{r_1\tilde{p}_0}{1+\frac{r_1-1}{C_m}(\tilde{p}_0+a_1^M\tilde{q}_0)}\int_{\mathbb{R}}k_1(x-y)\di y\leqslant \tilde{p}_0,\]
	which implies that
	\[\frac{r_1}{1+\frac{r_1-1}{C_m}(\tilde{p}_0+a_1^M\tilde{q}_0)}\leqslant 1.\]
	Similarily, we have $\displaystyle\frac{r_2}{1+\frac{r_2-1}{C_M}(\tilde{q}_0+a_2^m\tilde{p}_0)}\geqslant 1$. 
	A simple computation shows that
	\[\frac{\tilde{q}_0+a_2^m\tilde{p}_0}{C_M}\leqslant 1\leqslant \frac{\tilde{p}_0+a_1^M\tilde{q}_0}{C_m}, \]
	that is, 
	\begin{equation}\label{him}
	\Big(C_m-C_Ma_1^M\Big)\tilde{q}_0\leqslant \Big(C_M-C_ma_2^m\Big)\tilde{p}_0.
	\end{equation}
	By assumption (M), we obtain 
	\[C_m-C_Ma_1^M>0, \ C_M-C_ma_2^m\leqslant 0,\]
	which is a contradiction to (\ref{him}).
	
	Now we prove (H4). By Proposition \ref{rem} (ii) and (iii) with $m(x)=r_1(x)$, it is easy to see that the principle $\lambda_1(\mu)$ of \eqref{eep} is an even function of $\mu$ on $\mathbb{R}$. Since $\lambda_1(\mu)$ is $\ln$-convex on $\mathbb{R}$ and $\lambda_1(0)>1$, we have $\lambda_1(\mu)>1,\forall \mu>0.$ It follows that $c_{1+}^*=\inf\limits_{\mu>0}\frac{\ln\lambda _1(\mu)}{\mu}>0$. Similarly, we can show that $c_{2-}^*>0$. Thus, we have $c_{1+}^*+c_{2-}^*>0$.
	
	To verify (H5), it suffices to show that $\lim_{\mu\to 0^+}\frac{\ln\lambda _2(\mu)}{\mu}=0$, where $\lambda_2(\mu)$ is the principal eigenvalue of \eqref{eep2}. By Lemma \ref{rem}(b)(c), $\ln\lambda_2(\mu)$ is an even function on $\mathbb{R}$, and $n$-ordered differentiable (see \cite{Ding15,Liang}). Since $\lambda_2(0)=1$, it follows that
	$\lim_{\mu\rightarrow 0^+}\frac{\ln\lambda _2(\mu)}{\mu}=0<c^*_{1+}.$	
\end{proof}

As a consequence of Lemma \ref{M} and Theorem \ref{VLEQ}, we have the following result.
\begin{theorem}
	Assume that (K1)--(K3), and (M) hold. Then $E_1:=(p^*(x),0)$ is globally asymptotically stable with respect
	to initial values in $\mathbb{P}_+\backslash\{0,E_2\}$.
\end{theorem}   
For simplicity, we transfer system \eqref{dhmp} into the following cooperative system:
\begin{eqnarray}\label{dhmp2}
& &u_{n+1}(x)=
\int_{\mathbb{R}}\frac{r_1u_n(y)}{1+b_1(y)(u_n(y)+a_1(y)(q^*(y)-v_n(y))}k_1(x-y)\di y, \\ 
& &v_{n+1}(x)=
\int_{\mathbb{R}}\frac{r_2}{1+b_2(y)q^*(y)}\cdot \frac{b_2(y)a_2(y)q^*(y)u_n(y)+v_n(y)}{1+b_2(y)(q^*(y)+a_2(y)u_n(y)-v_n(y))}k_2(x-y)\di y.   \nonumber
\end{eqnarray}

By virtue of Propositions \ref{rem} and \ref{lb}, we see that $\overline{c}_+\ge c^0_+>0$. The next result about spreading speeds is implied by Theorem \ref{Qspreading}. 
\begin{theorem}
	Assume that (K1)--(K3), and (M) hold. Let $u(t,\cdot,\phi)$ be the solution of system \eqref{dhmp2} with $u(0,\cdot)=\phi\in\mathcal{C}_{u^*}$. Then the following statements are valid for system \eqref{dhmp2}:
	\begin{enumerate}
		\item[(i)]If $\phi\in\mathcal{C}_{\beta}$, $0\leqslant \phi\leqslant \omega\ll \beta$ for some $\omega\in \mathcal{C}^{per}_{\beta}$, and $\phi(x)=0, \forall x\ge H$, for some $H\in \mathbb{R}$, then $\lim\limits_{n\rightarrow\infty,x\ge cn}(u_n(x,\phi),v_n(x,\phi))=(0,0)$ for any $c>\overline{c}_+$.
		\item[(ii)]If $\phi\in\mathcal{C}_{\beta}$ and $\phi(x)\ge \sigma$, $\forall x\leqslant K$, for some $\sigma\in \mathbb{R}^2$ with $\sigma\gg0$ and $K\in\mathbb{R}$, then $\lim\limits_{n\rightarrow\infty,x\leqslant cn}((u_n(x,\phi),v_n(x,\phi))-\beta(x))=0$ for any $c<\overline{c}_+$.
	\end{enumerate}
\end{theorem}
In view of Theorem \ref{MIN}, we have the following result on periodic traveling waves for system \eqref{dhmp}. 
\begin{theorem}\label{her}
	Assume that (K1)--(K3), and (M) hold.
	Then for any $c\ge\overline{c}_+$, system \eqref{dhmp} has an L-periodic rightward traveling wave $(U(x-cn,x),V(x-cn,x))$ connecting $(p^*(x),0)$ to $(0,q^*(x))$ with the wave profile component $U(\xi,x)$  being continuous and non-increasing in $\xi$, and $V(\xi,x)$ being continuous and non-decreasing in $\xi$. While for any $c\in(0,\overline{c}_+)$, system \eqref{dhmp} admits no $L$-periodic rightward traveling wave connecting $(p^*(x),0)$ to $(0,q^*(x))$.
\end{theorem}

The above results shows that if $a_1^M<\displaystyle \frac{C_m}{C_M}<1<\frac {C_M}{C_m}<a_2^m $, i.e., $1$-th species is always a better and strong competitor, then $1$-th species can invade and futher replace $2$-th species in a osciallting wasy no matter what movement stragety is taken. Below we present some simulations results for the process of invasion. For this purpose, we truncate the infinite domain  $\mathbb{R}$ to a finite domian $[-M,M]$, where $M$ is sufficiently large. The evolution of the solution is shown in figures. Let $r_1=r_2=e$, 
\begin{align*}
a_1(x)=\left\{\begin{array}{l}
0.3, \ 0\leqslant x<5.5, \\[2mm]
0.4,  \ 5.5\leqslant x<10,
\end{array}\right.
a_2(x)=\left\{\begin{array}{l}
2, \ 0\leqslant x<5.5, \\[2mm]
1.5,  \ 5.5\leqslant x<10,
\end{array}\right.
\end{align*}
\begin{align*}
C(x)=\left\{\begin{array}{l}
1, \ 0\leqslant x<5.5, \\[2mm]
0.5,  \ 5.5\leqslant x<10.
\end{array}\right.
\end{align*}
Figure 1 shows that under the same type dispersal kernel, taking a small dispersal, i.e., trying to stay in the patch, cannot help to reduce the loss induced by the intracompetitin. Figure 2 shows that the success of invasion of $1$-th species into $2$-th species is independent of the particular type of dispersal kernels. 
\begin{figure*}[!htbp]
	\begin{center}
		\subfigure[The evolution of $p_n$ with d=1]{
			\includegraphics[width=0.475\textwidth]{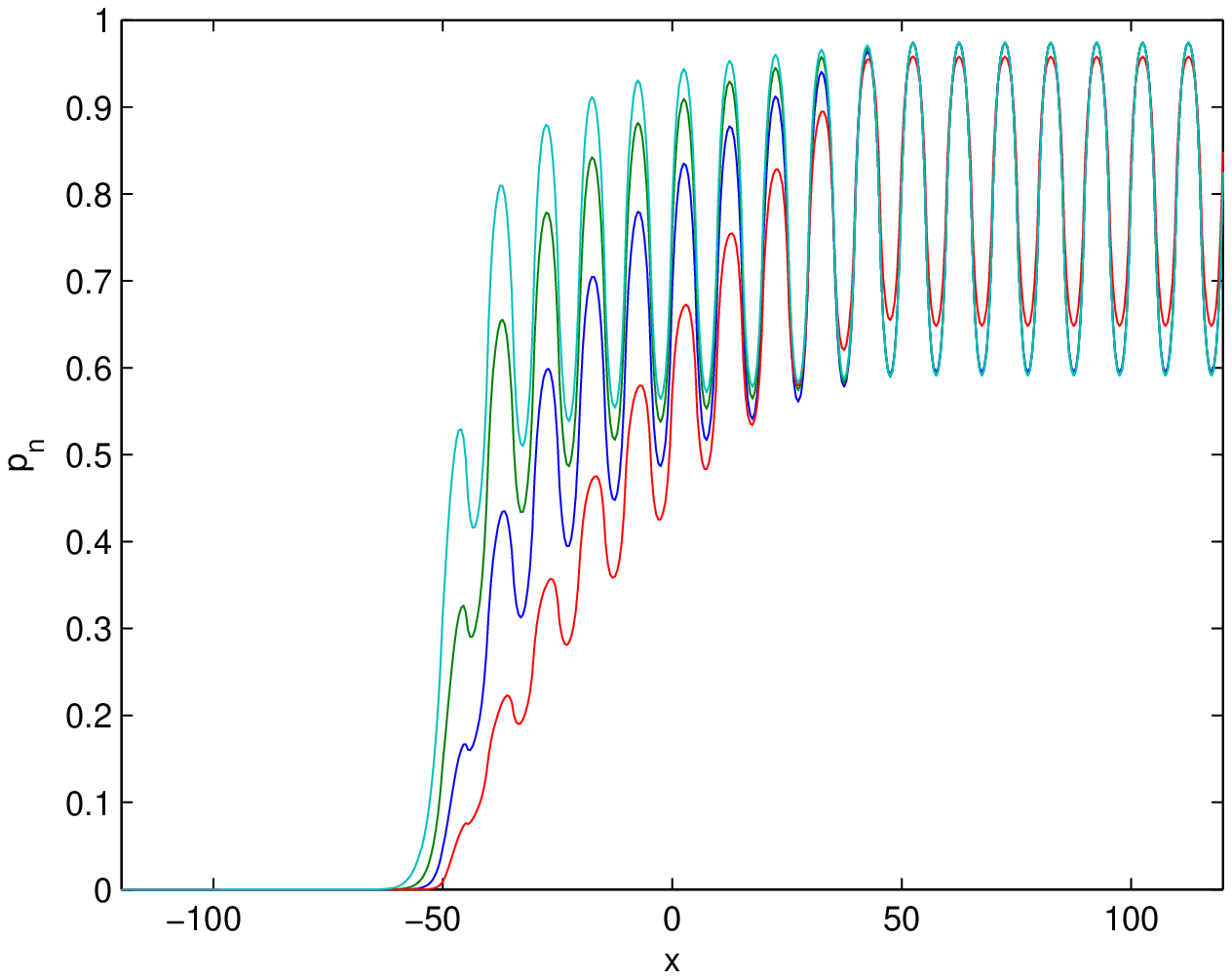} \label{p_n}
		}
		\subfigure[The evolution of $q_n$ with d=0.1]{
			\includegraphics[width=0.475\textwidth]{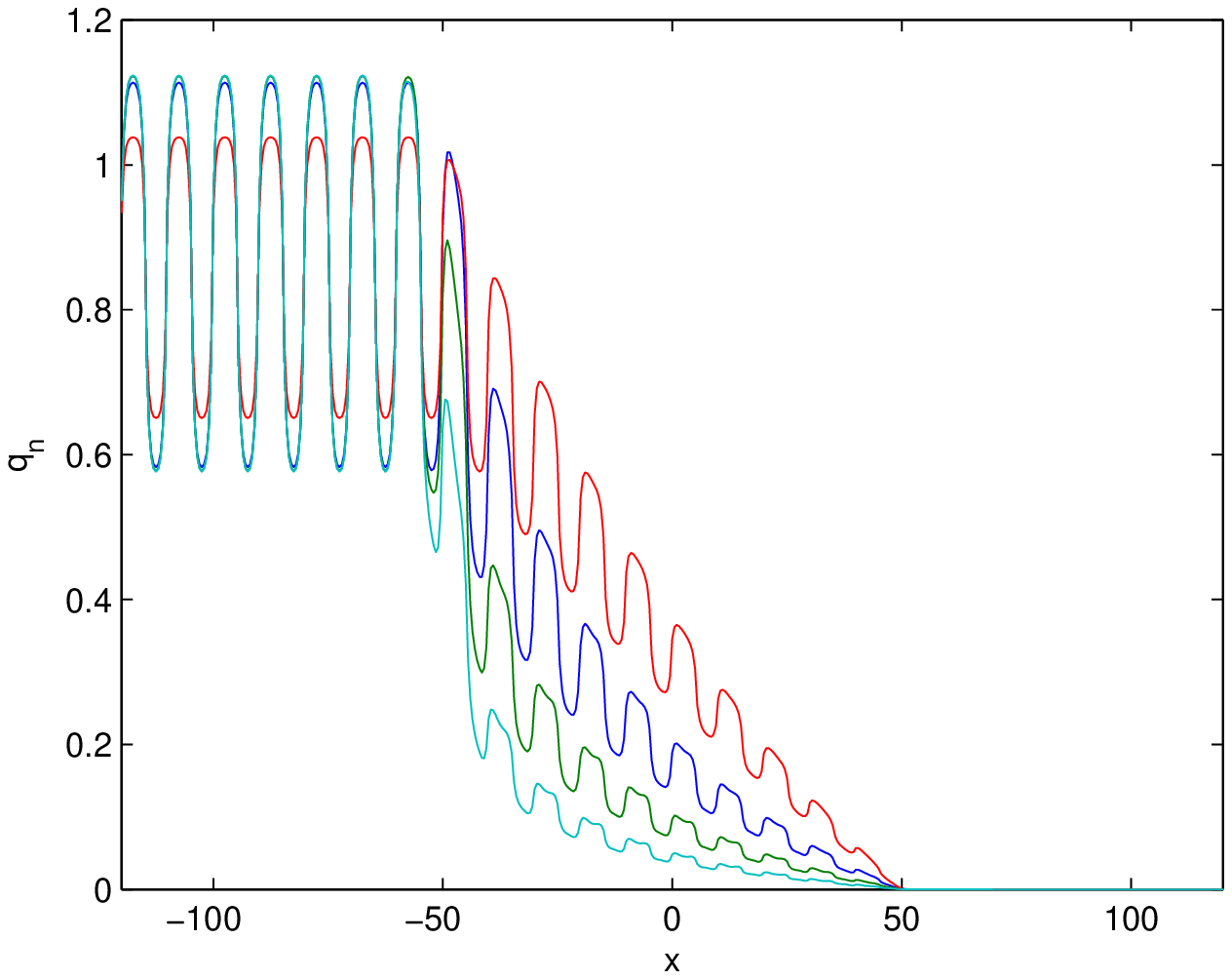} \label{q_n}
		}
		\caption[]{\label{fig} The evolution of $p_n$ and $q_n$ with a Laplace kernel, when $n=2,4,6,8$. }
	\end{center}
\end{figure*}

To obtain the linear determacy of $c^*$, we need to verify (D1) and (D2). Below we provide an example with simple senario where two species have same growh ability and competition ability, but their responses to environment changing are different. We assume that species-$1$, always have better response towards the varying environment conditions than  species-$2$, that is,  $C_1(x)>C_2(x)>0, \ \forall x \in \mathbb{R}$.
\begin{proposition}We consider the following spatially periodic competition model
	\begin{eqnarray}\label{dhmpcom}
	& &p_{n+1}(x)=
	\int_{\mathbb{R}}\frac{rp_n(y)}{1+b_1(y)(p_n(y)+q_n(y))}k(x-y)\di y, \\ 
	& &q_{n+1}(x)=
	\int_{\mathbb{R}}\frac{rq_n(y)}{1+b_2(y)(q_n(y)+p_n(y))}k(x-y)\di y, \ x\in\mathbb{R}, \nonumber
	\end{eqnarray}
	where 
	$b_i(x)=\displaystyle\frac{r-1}{C_i(y)}$, $C_i(x)$ is L-periodic with $C_1(x)>C_2(x)>0$, $a_1=a_2=1$ and $r>1$ are constant, $k(x-y)$ satisfies (K2)-(K3). (H1)-(H5) and (D1)-(D2) are valid.
\end{proposition}

\begin{proof}
		(H1) holds immedately by Proposition \ref{rem} (i) with $m_i(x)=r>1, \ (i=1,2)$.
		
		Now we verify (H2). Let $(0,q^*)$ be the $L$-periodic semi-trivial steady state of sytem (\ref{dhmp}), which is guaranteed by (H1), and $q_0=\max\limits_{x\in[0,L]}q^*(x)$. By a comparison argument, we have
		\[\int_{\mathbb{R}}\frac{rq_0}{1+\frac{r-1}{C_2(y)}q_0}k(x-y)\di y\geqslant \int_{\mathbb{R}}\frac{rq^*}{1+\frac{r-1}{C_2(y)}q^*}k(x-y)\di y=q^*.\]
		It then follows that
		\[\int_{\mathbb{R}}\frac{rq_0}{1+\frac{r-1}{C_2(y)}q_0}k(x-y)\di y\geqslant q_0,\]
		and hence,
		\[\int_{\mathbb{R}}\frac{rq_0}{1+\frac{r-1}{C^M_2}q_0}k(x-y)\di y
		\geqslant\int_{\mathbb{R}}\frac{rq_0}{1+\frac{r-1}{C^M_2(y)}q_0}k(x-y)\di y\geqslant q_0,\]
		that is,
		\[\frac{r}{1+\frac{r-1}{C^M_2}q_0}\geqslant 1, \ \text{i.e.,}\ q^*\leqslant q_0\leqslant C^M_2. \]
		Taking $m_1(x)=\displaystyle\frac {r}{1+b_1(x)q^*(x)}=\frac{r}{1+\frac{r-1}{C_1(y)}q^*(x)}$, we have 
		\[\frac{r}{1+\frac{r-1}{C_1(y)}q^*(x)}\geqslant \frac{r}{1+\frac{r-1}{C_2(y)}q^*(x)}\geqslant \frac{r}{1+\frac{r-1}{C^m_2}C^M_2}>\frac {r}{1+r-1}=1.\]
		Thus, $\lambda\Big(k, \displaystyle\frac {r}{1+b_1q^*}\Big)>1$ due to Proposition \ref{rem} (i).
		
		We prove (H3) by a way of contradiction. Suppose that $(\tilde{p},\tilde{q})$ is the positive $L$-periodic steady state. Since
		\begin{eqnarray}\label{coe}
		& &\tilde{p}(x)=
		\int_{\mathbb{R}}\frac{r\tilde{p}(y)}{1+b_1(y)(\tilde{p}(y)+\tilde{q}(y))}k(x-y)\di y, \\ 
		& &\tilde{q}(x)=
		\int_{\mathbb{R}}\frac{r\tilde{q}(y)}{1+b_2(y)(\tilde{q}(y)+\tilde{p}(y))}k(x-y)\di y, \ x\in\mathbb{R}, \nonumber
		\end{eqnarray}
		it follows that $$\lambda \Big(k,\frac {r}{1+b_1(\tilde{p}+\tilde{q})}\Big)=\lambda \Big(k,\frac {r}{1+b_2(\tilde{p}+\tilde{q})}\Big)=1.$$
		Note that $$\frac {r}{1+b_1(\tilde{p}+\tilde{q})}> \frac {r}{1+b_2(\tilde{p}+\tilde{q})}.$$
		Then Proposition \ref{rem} (i) implies that 
		 $$\lambda \Big(k,\frac {r}{1+b_1(\tilde{p}+\tilde{q})}\Big)>\lambda \Big(k,\frac {r}{1+b_2(\tilde{p}+\tilde{q})}\Big),$$
		 which is a contradiction.
		 
		 Assumptions (H4) and (H5) can be verified by arguments similar to those in the proof of Lemma \ref{M}.
		 
		 Condition (D1) is easy to verify. Since $1+b_1q^*<(1+b_1q^*)^2<(1+b_2q^*)^2$ with $C_1(x)>C_2(x), \ \forall x\in \mathbb{R}$, then for enigenvalue problems \eqref{eep0} and \eqref{Lpep2}, we have $\lambda_0(\mu)>\overline{\lambda}(\mu)$ due to Proposition \ref{rem} (i).
		 
		 Regarding condition (D2), since $a_1=a_2=1$, we have $\max\{a_1(x),\frac{1}{a_{2}(x)}\}=1$. Let $\phi^*=(\phi^*_1,\phi^*_2)$ be the associated positive eigenfunctions associated with the principle eigenvalue $\lambda$ of the periodic eigenvalue problem \eqref{Lpep}. SInce 
		 	\begin{align*}
		 	&\lambda \phi^*_1-\int_{\mathbb{R}}\frac{r}{1+b_2(y)q^*(y)}\cdot \frac{b_2(y)q^*(y)\phi_1^*(y)+\phi^*_1(y)}{1+b_2(y)q^*(y)}e^{-\mu(y-x)}k(x-y)\di y,\\
		 	&=\lambda \phi^*_1-\int_{\mathbb{R}}\frac{r}{1+b_2(y)q^*(y)}\phi_1^*(y)e^{-\mu(y-x)}k(x-y)\di y\\
		 	&\geqslant\lambda \phi^*_1-\int_{\mathbb{R}}\frac{r}{1+b_1(y)q^*(y)}\phi_1^*(y)e^{-\mu(y-x)}k(x-y)\di y\\
		 	&=\lambda \phi^*_1-\lambda \phi^*_1=0,
		 	\end{align*}
		 	it follows that $\phi^*_1$ is a upper solution, and hence,  
		 	$\phi^*_1\geqslant \phi^*_2$, i.e., $\frac {\phi^*_1}{\phi^*_2}\geqslant 1$.
\end{proof}

To finish this paper, we remark that in the case where system \eqref{VL} admits a unique positive steady state, its spatial dymamics is relatively simple from the viewpoint of mathematical analysis, as we can apply the theory developed in \cite{Liang2,Wein02} directly to the existence of two different spatially periodic travelling waves connecting $(0,q^*)$ and $(\tilde{p},\tilde{q})$, $(p^*,0)$ and $(\tilde{p},\tilde{q})$, respectively, under appropriate conditions.

\begin{figure*}[!htbp]
	\begin{center}
		\subfigure[The evolution of $p_n$ with a Gaussian kernel.]{
			\includegraphics[width=0.475\textwidth]{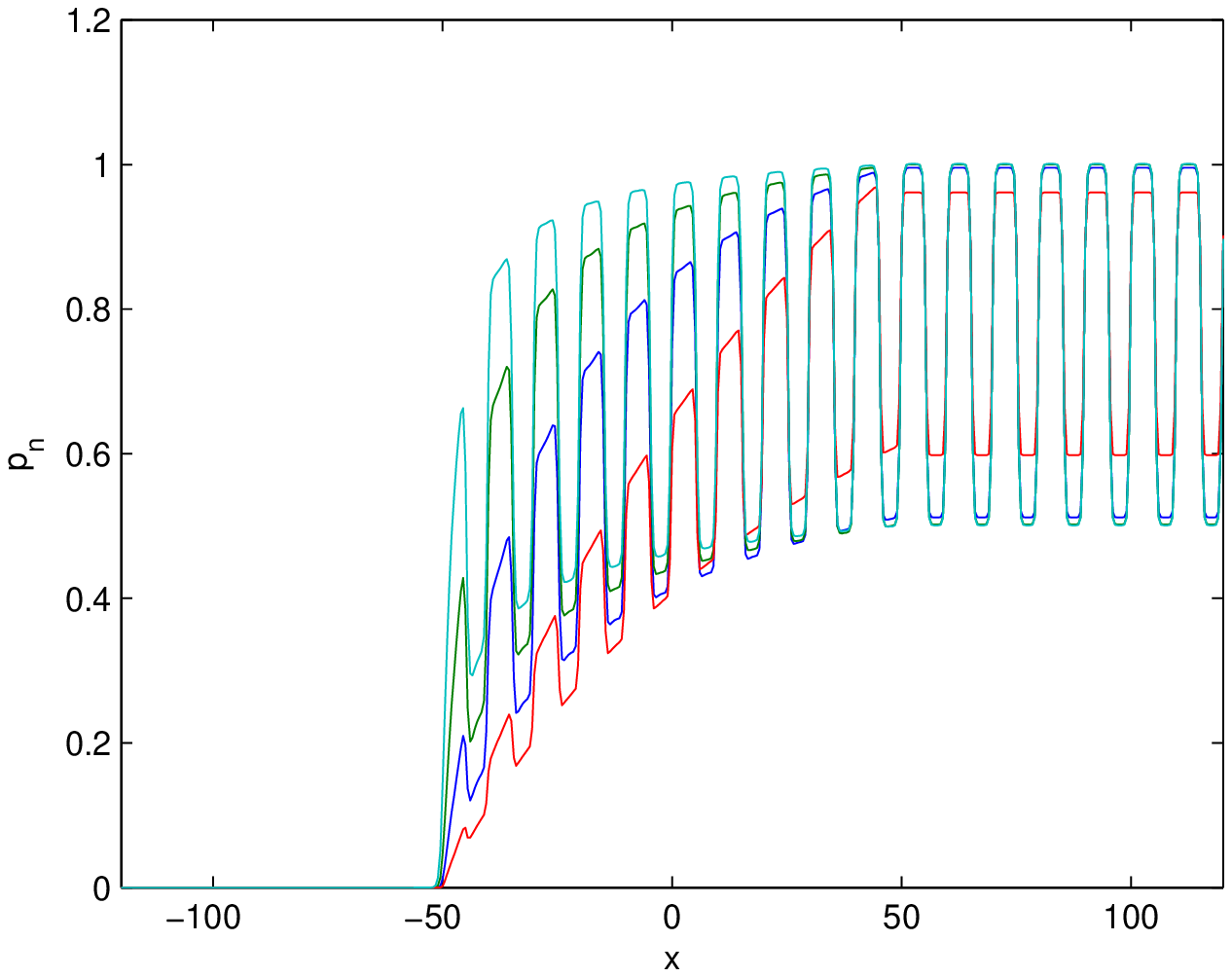} 
		}
		\subfigure[The evolution of $q_n$ with a Laplace kernel.]{
			\includegraphics[width=0.475\textwidth]{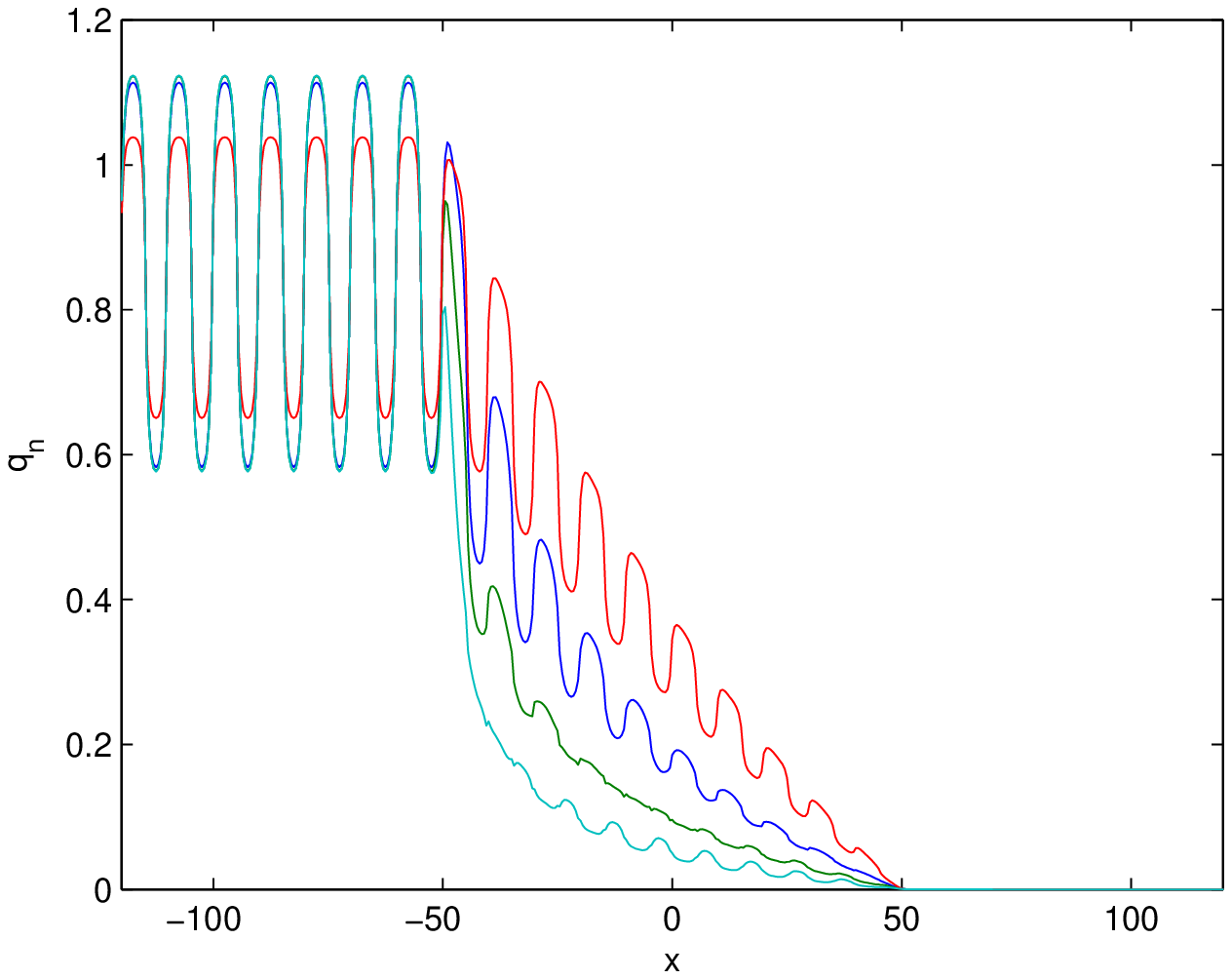} 
		}
		\caption[]{\label{fig} The evolution of $p_n$ and $q_n$ with $k_1(x-y)=\frac{1}{\sqrt{2\pi\times0.1}}e^\frac{-(x-y)^2}{0.2}$, $k_2(x-y)=\frac{1}{2\times0.5}e^\frac{-|x-y|}{0.5}$, when $n=2,4,6,8$.
		}
	\end{center}
\end{figure*}

\section*{Acknowledgement}
We are grateful to Drs. Frithjof Lutscher and Xiao Yu for helpful discussions and valuable comments.

\end{document}